\documentclass{amsart}
\usepackage{amsmath}
\usepackage{paralist}
\usepackage{amsfonts}
\usepackage{amssymb}
\usepackage{amscd}
\usepackage{amsrefs}

\usepackage{float}
\usepackage{tikz}
\usepackage{graphicx}
\usepackage[colorlinks=true]{hyperref}
\hypersetup{urlcolor=blue, citecolor=red}

\newtheorem{theorem}{Theorem}[section]
\newtheorem{corollary}{Corollary}

\newtheorem{lemma}[theorem]{Lemma}

\theoremstyle{definition}
\newtheorem{definition}[theorem]{Definition}

\title[Generalized Strichartz estimates] %Use the shortened version of the full title
      {Generalized inhomogeneous Strichartz estimates}

\author[Robert Schippa]{Robert Schippa}
\address{Fakult\"at f\"ur Mathematik, Universit\"at Bielefeld, Postfach 10 01 31, 33501 Bielefeld, Germany}
\subjclass{Primary: 35B45; Secondary: 35Q55.}
 \keywords{dispersive equations, Strichartz estimates, inhomogeneous equations, spherical symmetry, spherical averages}
 \email{robert.schippa@uni-bielefeld.de}

\begin{document}
\maketitle

\bigskip
\begin{abstract}
We prove new inhomogeneous generalized Strichartz estimates, which do not follow from the homogeneous generalized estimates by virtue of the Christ-Kiselev lemma. Instead, we make use of the bilinear interpolation argument worked out by Keel and Tao and refined by Foschi presented in a unified framework. Finally, we give a sample application.
\end{abstract}

\section{Introduction}
\label{section:introduction}
We start by briefly revisiting homogeneous and inhomogeneous Strichartz estimates, their link and previous results.
 For unexplained terminology and notation see Section \ref{section:preliminaries}.
Let us consider the homogeneous equation
 \begin{equation}
 \label{eq:homogeneousEquation}
 \left\{\begin{array}{cl}
 i \partial_t u(t,x) + \varphi(D) u(t,x) &= 0, \, (t,x) \in \mathbb{R} \times \mathbb{R}^n, \\
 u(0,\cdot) &= u_0.  \end{array} \right.
 \end{equation}
 In the following we shall confine ourselves to the dispersion relation $\varphi(\rho) = \rho^a, \, a \geq 1$, which yields a simple scaling condition for the free solutions. The method can be extended to more general phase functions see e.g. \cite{Ovcharov2012} but since the generalization is straight-forward we choose to focus on the main argument.\\
We shall denote the unitary group generated by $iD^a$ as $$\mathcal{U}_a = \left( U_a(t) \right)_{t \in \mathbb{R}} = \left( e^{it D^a} \right)_{t \in \mathbb{R}}.$$
 The homogeneous Strichartz estimates control the mixed $L_t^q L_x^p$-norm of the free solution with respect to an $L^2$-Sobolev norm of the initial datum:
 \begin{equation}
 \label{eq:homogeneousStrichartzEstimate}
 \Vert e^{it D^a} u_0 \Vert_{L_t^q L_x^p} \lesssim_{n,p,q} \Vert u_0 \Vert_{\dot{H}^{-s}}
 \end{equation}
 with $-s=\frac{n}{2} - \frac{n}{p} - \frac{a}{q}$ fixed by scaling and $q,p \geq 2$ due to translation invariance.\\
 For classical Strichartz estimates, that are estimates, which hold without further assumptions on the wave-functions, the sharp range of the homogeneous estimates was found in \cite{TaoKeel1998} by Keel and Tao starting from an energy estimate and a dispersive estimate:\\
  When $\mathcal{U}=(U(t))_{t \in \mathbb{R}}$ denotes the propagator, typically after localizing frequencies to unit scale,  the energy estimate states as
 \begin{equation*}
 \Vert U(t) u_0 \Vert_{L_x^2} \lesssim_{n} \Vert u_0 \Vert_{L_x^2}
 \end{equation*}
 and the dispersive estimate can come up as the untruncated decay
 \begin{equation}
 \label{eq:untruncatedDecayEstimate}
 \Vert U(t) u_0 \Vert_{L_x^\infty} \lesssim_n |t|^{-\sigma} \Vert u_0 \Vert_{L_x^1} \; \; (t \neq 0)
 \end{equation}
 or as the (stronger) truncated decay
 \begin{equation}
 \label{eq:truncatedDecayEstimate}
 \Vert U(t) u_0 \Vert_{L_x^\infty} \lesssim_n (1+|t|)^{-\sigma} \Vert u_0 \Vert_{L_x^1} \; \; (t \neq 0).
 \end{equation}
 $\sigma$ is called the decay parameter, which we find to be
\begin{equation*}
\label{eq:decayParametersUa}
\sigma(a,n)= \left\{\begin{array}{cl} \frac{n-1}{2}, &\; \mbox{if } a =1, \\
 \frac{n}{2}, &\; \mbox{if } a \neq 1. \end{array} \right. 
 \end{equation*}
For $a \neq 1$ see for instance \cite[Remark~2,~p.~1644]{Guo2008}, for $a=1$ this is common knowledge. Since we shall work at fixed spatial dimension and with a fixed unitary group, we will usually suppress the dependence of $a$ and $n$.
When we consider local inhomogeneous estimates in Section \ref{subsection:localInhomogeneousEstimates} we state our modified assumptions on decay estimates.
For further references on the history of Strichartz estimates we also refer to \cite{TaoKeel1998} and references therein. The sharp range is found by maximally anisotropically propagating waves, so-called Knapp-type examples (cf. \cite[p.~964]{TaoKeel1998}). More estimates become available, e.g. if one considers spherically symmetric data ruling out the classical Knapp-type examples. For Strichartz estimates for more general dispersion relations with non-vanishing second derivative, e.g. $\varphi(\rho) = (1+\rho^2)^{1/2}$ which relates to the Klein-Gordon equation we refer to \cite{ChoOzawaXia2011}. It turns out that the admissible range is the same as for Schr\"odinger-like equations, although the corresponding estimate \eqref{eq:homogeneousStrichartzEstimate} involves a pseudo-differential operator taking into account the inhomogeneity of the dispersion relation.\\
 Inhomogeneous estimates come into play controlling the solution to the inhomogeneous equation with zero-initial condition
 \begin{equation*}
 \label{eq:inhomogeneousEquation}
 \left\{\begin{array}{cl}
 i \partial_t u(t,x) + D^a u(t,x) = F(t,x), \, (t,x) \in \mathbb{R} \times \mathbb{R}^n, \\
 \lim_{t \rightarrow - \infty} u(t,x) = 0. \end{array} \right.
 \end{equation*}
 The weak solution is given by the Duhamel formula
 \begin{equation*}
 \label{eq:weaksolutionInhomogeneousEquation}
 u(t,x) = -i \int_{-\infty}^t e^{i(t-\tau)D^a} F(\tau, x) d\tau
\end{equation*}
and inhomogeneous estimates state as follows:
\begin{equation*}
\label{eq:inhomogeneousEstimates}
\left\Vert \int_{-\infty}^t e^{i(t-\tau)D^a} F(\tau) d\tau \right\Vert_{L_t^{\tilde{q}} L_x^{\tilde{p}}} \lesssim_{n,p,q,\tilde{p},\tilde{q}} \Vert D^{-2s} F \Vert_{L_t^{q^\prime} L_x^{p^\prime}}
\end{equation*} 
This time we have the scaling condition:
\begin{equation}
\label{eq:scalingInhomogeneousEstimates}
\beta_a(q,\tilde{q},p,\tilde{p},s)=\frac{1}{q} + \frac{1}{\tilde{q}} - \frac{n}{a} \left( 1- \frac{1}{p} - \frac{1}{\tilde{p}} \right) - \frac{2s}{a} = 0
\end{equation}
Denoting the time-evolution operator as
\begin{equation*}
\begin{split}
T: L^2 \rightarrow& L_t^q L_x^p \\
	u_0 \mapsto& D^{s} e^{it D^a} u_0,
	\end{split}
\end{equation*}
we find the adjoint operator to be
\begin{equation*}
\begin{split}
T^*: L_t^{\tilde{q}^\prime} L_x^{\tilde{p}^\prime} \rightarrow& L^2 \\
	F \mapsto& D^{\tilde{s}} \int_{-\infty}^{\infty} e^{-i \tau D^a} F(\tau) d\tau,
\end{split}
\end{equation*}
and finally, we have
\begin{equation*}
\begin{split}
TT^*: L_t^{\tilde{q}^\prime} L_x^{\tilde{p}^\prime} \rightarrow& L_t^q L_x^p \\
		F \mapsto& D^{(s+\tilde{s})} \int_{-\infty}^{\infty} e^{i(t-\tau)D^a} F(\tau) d\tau.
\end{split}
\end{equation*}
From the Christ-Kiselev lemma (cf. \cite[Theorem~1.2,~p.~410]{Christ2000}) and its operator-valued extension (cf. \cite[pp.~1481-1483]{Tao2000}) we find that two homogeneous estimates 
with coefficients $(q,p)$ and $(\tilde{q},\tilde{p})$ yield an inhomogeneous estimate with coefficients $(q,p,\tilde{q},\tilde{p})$ if $\tilde{q}^\prime < q$.\\
A precise analysis performed by Foschi in \cite{Foschi2005} showed that the method employed in \cite{TaoKeel1998} can be extended to find more inhomogeneous estimates than the ones, which already follow from the homogeneous estimates and the Christ-Kiselev lemma. We shall see that this method is not confined to classical Strichartz estimates but that one can also start with a generalized setting. The important notions will be declared in greater detail in Section \ref{section:preliminaries}.
In the following we shall work with the notion of range spaces $Z^s_p$, which resemble $L^p$-spaces with derivatives but incorporate the additional assumptions on the wave-functions, effectively giving rise to an extended range.
In the special case of spherical symmetry this has already been done in \cite{Ovcharov2012}, though not with the sharp range of decay parameters.\\
We shall consider global estimates of the kind
\begin{equation}
\label{eq:globalEstimateI}
\left\Vert \int_{-\infty}^t e^{i(t-\tau)D^a} F(\tau) d\tau \right\Vert_{L_t^q Z^{s}_p} \lesssim_{n,p,q,\tilde{p},\tilde{q}} \Vert F \Vert_{L_t^{\tilde{q}^\prime} Z^{-s}_{\tilde{p}^\prime}}
\end{equation}
and
\begin{equation}
\label{eq:globalEstimateII}
\left\Vert \int_{-\infty}^t e^{i(t-\tau)D^a} F(\tau) d\tau \right\Vert_{Z^{s}_{p,q}} \lesssim_{n,p,q,\tilde{p},\tilde{q}} \Vert F \Vert_{Z^{-s}_{\tilde{p}^\prime, \tilde{q}^\prime}}.
\end{equation}
We find the following theorem to hold:
 \begin{theorem}[Global inhomogeneous estimates]
 \label{thm:globalInhomogeneousEstimates}
Let $a \geq 1$. Suppose that the family of linear operators $\mathcal{U}_a$ admits generalized homogeneous Strichartz estimates with range spaces $(Z_p)_{p \in [1,\infty]}$, extended decay parameter $\sigma^\prime$ and with the generalized Strichartz estimates admitting a generalized dispersive estimate. Suppose that for $1 \leq q, \tilde{q}, p, \tilde{p} \leq \infty$, $s \in \mathbb{R}$ we have $\beta_a(q,\tilde{q},p,\tilde{p},s)=0$. In the non-sharp case, that is  $1/q + 1/\tilde{q} < 1 , \, q, \, \tilde{q} < \infty$, we find the estimates \eqref{eq:globalEstimateI} and \eqref{eq:globalEstimateII} to hold, if 
 \begin{equation*}
 \begin{split}
\exists \; \sigma_1, \sigma_2 \in (\sigma,\sigma^\prime):\\
1 \leq &\mu = \frac{(a/2) \left( \sigma_1/2 + \sigma_2/2 - \sigma \right)}{s-r + \left( (a \sigma_1 -n)/p + (a \sigma_2 -n)/\tilde{p} \right)/2} < \infty \\
\frac{\sigma_1 - 1}{\sigma_1} \frac{p}{2} \leq \mu \leq \frac{p}{2} &\; \; \; \; \frac{\sigma_2 -1}{\sigma_2} \frac{\tilde{p}}{2} \leq \mu \leq \frac{\tilde{p}}{2} \nonumber \\
\frac{(\sigma_1/2)}{(1/q) + (\sigma_1/p)} < \mu &\; \; \; \; \frac{(\sigma_2/2)}{(1/\tilde{q}) + (\sigma_2/\tilde{p})} < \mu \nonumber
\end{split}
\end{equation*}
In the sharp case, that is $1/q + 1/\tilde{q} = 1 $, where, in addition to the requirements for the non-sharp case $2 < p, \tilde{p} < \infty$, we find the estimate \eqref{eq:globalEstimateII} to hold, if
 \begin{align*}
\exists \; \sigma_1, \sigma_2 \in (\sigma,\sigma^\prime):\\
1 < &\mu = \frac{(a/2) \left( \sigma_1/2 + \sigma_2/2 - \sigma \right)}{s-r + \left( (a \sigma_1 -n)/p + (a \sigma_2 -n)/\tilde{p} \right)/2} < \infty \\
\frac{\sigma_1 - 1}{\sigma_1} \frac{p}{2} < \mu < \frac{p}{2} &\; \; \; \;
\frac{\sigma_2 -1}{\sigma_2} \frac{\tilde{p}}{2} < \mu < \frac{\tilde{p}}{2} \\
\frac{(\sigma_1/2)}{(1/q) + (\sigma_1/p)} < \mu &\; \; \; \; \frac{(\sigma_2/2)}{(1/\tilde{q}) + (\sigma_2/\tilde{p})} < \mu \\
\frac{1}{p} \leq \frac{1}{q} &\; \; \; \; \frac{1}{\tilde{p}} \leq \frac{1}{\tilde{q}}
 \end{align*}
 \end{theorem}
 \begin{figure}[H]
 \begin{tikzpicture}[>=stealth,scale=0.5]
 	\node[shape=circle,inner sep=2pt, minimum size=2pt,label=left:$O$,draw]		(O) at (0,0) {};
 	\node[shape=circle,inner sep=2pt, minimum size=2pt,label=above:$A$,draw]	(A) at (4,7) {};
 	\node[shape=circle,inner sep=2pt, minimum size=2pt,label=right:$C$,draw]	(C) at (7,3) {};
 	\node[shape=circle,inner sep=2pt, minimum size=2pt,label=left:$B$,draw]		(B) at (4,3) {};
 	\node[shape=circle,inner sep=2pt, minimum size=2pt,label=above:$D$,draw]	(D) at (7,7) {};
 	\draw[->, very thick]	(O)-- (8,0) node[anchor=north] {$\frac{1}{p}$} ;
 	\draw[->, very thick]	(O) -- (0,8) node[anchor=east] {$\frac{1}{\tilde{p}}$};
 	\draw		(A) -- (B) -- (C) -- (D) -- (A) --(O) -- (C);
% 	\draw		(A) -- (K) -- (B);
 	\draw[dashed] (B) -- (4,0) node[anchor=north] {$\frac{\sigma_1 -1}{2\sigma_1}$};
 	\draw[dashed] (B) -- (0,3) node[anchor=east] {$\frac{\sigma_2 -1}{2\sigma_2}$};
 	\draw[dashed] (A) -- (0,7) node[anchor=east] {$\frac{1}{2}$};
 	\draw[dashed] (C) -- (7,0) node[anchor=north] {$\frac{1}{2}$};
 \end{tikzpicture}
 \caption{This pictorial representation generalizes \cite[Figure~2,~p.~5]{Foschi2005}. The axes refer to the spatial integrability coefficients. The rectangle $ABCD$ corresponds to estimates found from factorization and the application of the Christ-Kiselev lemma up to endpoints. The origin relates to the dispersive estimate; one finds local estimates to hold in the wedge $AOCD$ by virtue of interpolation, restrictions on global estimates cut off estimates with too large spatial integrability coefficients.}
\label{fig:globalInhomogeneousEstimates}
 \end{figure}
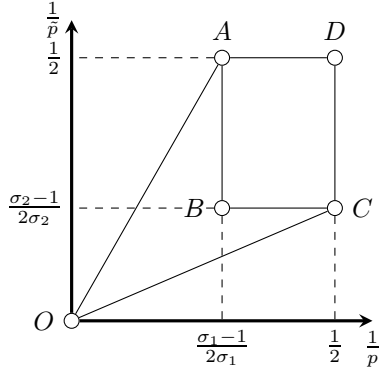
 We shall see in Section \ref{section:applications} that spherical symmetry and taking spherical averages yield generalized Strichartz estimates.
The main purpose of this article is to show that one can prove additional inhomogeneous generalized Strichartz estimates from homogeneous estimates in a unified framework. By additional inhomogeneous estimates we mean that these estimates do not follow from the homogeneous estimates and the Christ-Kiselev lemma. In the special case of additional homogeneous estimates stemming from spherical symmetry (cf. Section \ref{subsection:sphericalSymmetry}) this had been carried out previously in \cite{Ovcharov2012}. Our results extend the ones in \cite{Ovcharov2012} for Schr\"odinger-like equations because we work with the up to endpoints sharp range of homogeneous Strichartz estimates for spherically symmetric functions.
 We also find additional inhomogeneous estimates after taking spherical averages as an instance of Theorem \ref{thm:globalInhomogeneousEstimates} in Section \ref{subsection:sphericalAverages}. In this case our results appear to be completely new.\\
  The additional inhomogeneous estimates found after taking spherical averages can be applied to find a new well-posedness result for the fractional Schr\"odinger equation with time-dependent potential we will establish in Section \ref{subsection:fractionalSEQApplication}, where the following corollary provides the required additional estimates.
 \begin{corollary}
 \label{cor:inhomogeneousEstimatesApplication}
 There is some $\delta >0$, so that for $2-\delta \leq a<2$ there are coefficients $(q,q,\tilde{q},\tilde{q})$ which fulfill the requirements of Theorem \ref{thm:globalInhomogeneousEstimates} with a vanishing derivative parameter in a full neighbourhood of $q=\frac{2(n+a)}{n}$ if
 $ \frac{1}{q} + \frac{1}{\tilde{q}} = \frac{n}{n+a}. $\\
 In particular, we find the estimate 
 $$ \left\Vert \int_{-\infty}^t e^{i(t-\tau)D^a} F(\tau) d\tau \right\Vert_{L_t^q \mathcal{L}_r^q L_\omega^2} \lesssim_{n,q,\tilde{q}} \left\Vert F \right\Vert_{L_t^{\tilde{q}^\prime} \mathcal{L}_r^{\tilde{q}^\prime} L_\omega^2} $$
 to hold under the above assumptions.
 \end{corollary}
  The fractional Schr\"odinger equation with potential was also considered in \cite{Cho2016} though only for spherically symmetric potentials and solutions. The main ingredient for the proof of \cite[Theorem~1.2,~p.~1908]{Cho2016} were additional inhomogeneous estimates for spherically symmetric solutions. We will recap the proof with slight modifications to see how the additional inhomogeneous estimates found after taking spherical averages from Corollary \ref{cor:inhomogeneousEstimatesApplication} allow us to drop assumptions on spherical symmetry.
Due to the perturbative nature of the range of the admissible coefficients we choose not to state the  integrability conditions explicitly, which was carried out in \cite{Cho2016}. We remark that the range provided by Corollary \ref{cor:inhomogeneousEstimatesApplication} is significantly smaller than the one from \cite[Corollary~1.1,~p.~1907]{Cho2016}.\\
Note that the situation is very different when one considers a Schr\"odinger equation with a potential which is time-independent. Instead of perceiving the solution as perturbation of the homogeneous equation via Duhamel's formula one typically proceeds by proving the dispersive estimate for the generator of the full time-evolution (cf. \cite{Mizutani2016, BoucletMizutani2016, RodnianskiSchlag2004} and the references therein).\\
In specific cases this method extends to time-dependent potentials (cf. \cite{RodnianskiSchlag2004}) but the potentials under consideration in Section \ref{subsection:fractionalSEQApplication} are in general not compatible.
\section{Preliminaries}
\label{section:preliminaries}
\subsection{Notation}
\label{subsection:notation}
In this section we explain basic notation which we will employ throughout the text. Further definitions which demand more explanation can be found in the next sections.\\
Let $(X,\mu)$ be a measure space and $E$ a Banach space. For $q \in [1,\infty)$ we consider
\begin{equation*}
\begin{split}
L^q(X,E) &= \left\{ f: X \rightarrow E, \; \mu\mbox{-measurable} \, | \, \Vert f \Vert_q < \infty \right\} \\ 
 \Vert f \Vert_q &= \left( \int_{X} \Vert f(x) \Vert_E^q d\mu(x) \right)^{1/q}
\end{split}
\end{equation*}
with the usual modification for $q = \infty$.\\
We shall work under the general assumption that the estimates we prove are supposed to hold merely for sufficiently smooth and decaying initial data. Therefore, we will typically ignore any questions on measurability, which become more delicate if the space $L^\infty(E)$ is involved. The smoothness and decay assumptions allow us to be a bit careless regarding vector-valued integration and interpolation, see also \cite[Theorem~5.1.2.,~p.~107]{Bergh1976}.
% We ignore any questions of measurability.\\
In the special case $(X,\mu) = (\mathbb{R}^n,\lambda)$ and $E = L^p(\mathbb{R}^n,\mathbb{C}) =: L^p(\mathbb{R}^n) =:L^p$, if there is no ambiguity about spatial dimension, we have
\begin{equation*}
\Vert f \Vert_{L_t^q(\mathbb{R}, L_x^p(\mathbb{R}^n))} = \left( \int_{\mathbb{R}} \left( \int_{\mathbb{R}^n} |f(t,x)|^p dx\right)^{q/p} dt \right)^{1/q}.
\end{equation*}
We denote $L_t^q(\mathbb{R}, L_x^p(\mathbb{R}^n)) =: L_t^q L_x^p$.\\
We work with the following convention of the Fourier transform for $f \in \mathcal{S}(\mathbb{R}^n)$
\begin{equation*}
\hat{f}(\xi) = \int_{\mathbb{R}^n} f(x) e^{-i x \xi} dx,
\end{equation*}
which extends by standard means into the space of tempered distributions $\mathcal{S}^\prime(\mathbb{R}^n)$. For details on some basic assertions which are made in the following without further comment see e.g. \cite[Chapter~0,~pp.~1-38]{Sogge1993}.\\
We denote the sphere embedded into Euclidean space with $L^p$-spaces defined via surface measure $d \sigma$
\begin{equation*}
\mathbb{S}^{n-1} = \left\{ x \in \mathbb{R}^n \, | \, \Vert x \Vert_2 = 1 \right\}, \;  L^p(\mathbb{S}^{n-1},d\sigma) =: L_\omega^p.
\end{equation*}
Once again, we make use of the latter variation, if there is no confusion about spatial dimension.\\
We also introduce a shorthand notation for the radial part of functions in Euclidean space, that is $\mathcal{L}_r^p := L^p((0,\infty), \, r^{n-1} dr)$.\\
We define homogeneous and inhomogeneous Sobolev spaces in Euclidean space by powers of the operators $D=(-\Delta)^{1/2}$ and $\Lambda = (1-\Delta)^{1/2}$ as
\begin{equation*}
\begin{split}
\dot{H}^s(\mathbb{R}^n) = \left\{ f \in \mathcal{S}^\prime(\mathbb{R}^n)/\mathcal{P} \, | \, \Vert f \Vert_{\dot{H}^s} < \infty \right\}, \; \Vert f \Vert_{\dot{H}^s} = \Vert D^s f \Vert_{L^2}, \\
H^s(\mathbb{R}^n) = \left\{ f \in \mathcal{S}^\prime(\mathbb{R}^n) \, | \, \Vert f \Vert_{H^s} < \infty \right\}, \; \Vert f \Vert_{H^s} = \Vert \Lambda^s f \Vert_{L^2},
\end{split}
\end{equation*}
where $\mathcal{P}$ denotes the set of polynomials (i.e. tempered distributions with Fourier support concentrated at the origin).\\
In the following let $\psi:\mathbb{R}^n \rightarrow \mathbb{R}$ denote a fixed smooth spherically symmetric and in terms of the radial variable monotonically decreasing function satisfying $\psi(x) =1, \, |x| \leq 1$ and $\psi(x) = 0, \, |x| \geq 2$. We set $\chi:\mathbb{R}^n \rightarrow \mathbb{R}, \; \chi(x)= \psi(x)-\psi(2x)$ to define a suitable bump function with support around $1$. For $f \in \mathcal{S}^\prime(\mathbb{R}^n)$ we define the frequency localization operators
\begin{equation*}
(P_N f)^{\hat{\,}}(\xi) = \chi(\xi/N) \hat{f}(\xi).
\end{equation*}
Capital letters will denote dyadic numbers and we will use the notation $\tilde{P}_N = P_{N/2} + P_N + P_{2N}$.\\
%, note that $\tilde{P}_N P_N = P_N \tilde{P}_N = P_N$.\\
For Besov spaces we follow the conventions of \cite{Grafakos2009} and for angular derivatives we make use of the following notation: For $1 \leq i < j \leq n$ we set $\Omega_{ij}= i(x_i \partial_j - x_j \partial_i)$ to denote the generators of rotations, $\Delta_\omega = \sum_{i<j} \Omega_{ij}^2$ denotes the Laplace-Beltrami operator on the sphere extended to Euclidean space and $\Lambda_\omega = (1-\Delta_\omega)^{1/2}$ denotes the inhomogeneous angular derivative. For basic results see e.g. \cite{Sterbenz2005, Jiang2012}.\\
In estimates we use the notation $C_{a,b,\ldots} = C(a,b,\ldots)$, indicating that the generic constant $C$ depends only on the parameters $a,b,\ldots$. We also employ the shorthand notation $\lesssim_{a,b,\ldots}$. The constant is allowed to change at each occurrence, though.  More sophisticated dependencies will be mentioned properly.\\
\subsection{Setup}
\label{subsection:setup}
We consider generalized Strichartz estimates beyond the classical range originating from spherical symmetry or from weakening integrability in the spherical coordinates. We shall start from some basic assumptions on the function spaces we are working with and generalized homogeneous Strichartz estimates.\\
In order to generically write up the homogeneous generalized Strichartz estimates, let $(Z_p)_{p \in [1,\infty]}$ denote a family of Banach spaces of tempered distributions in $\mathbb{R}^n$. In the following let $n$ denote the fixed spatial dimension; we are not interested in comparing estimates for different spatial dimensions, but since decay parameters and scaling also depend on the spatial dimension, we keep track of it.\\
We make the following definition, so that the function spaces under consideration behave reasonably under frequency localization and scaling:
\begin{definition}[Compatibility property]
We say that the family $(Z_p)_{p \in [1,\infty]}$ of function spaces has the compatibility property, if we have
\begin{enumerate}
\item[(i)] the continuous embedding $Z_p \hookrightarrow \mathcal{S}^\prime(\mathbb{R}^n)$,
\item[(ii)] the continuity of frequency localization:
\begin{equation}
\label{eq:compFrequLoc}
\forall R>0, \, p \in [1,\infty]: \, \Vert P_R \Vert_{Z_p \rightarrow Z_p} < C_n, \; \left(P_R f\right)\hat{\,}(\xi) = \chi(\xi/R) \hat{f}(\xi),
\end{equation}
\item[(iii)] and a vector-valued $L^p$-structure: There is a separable Hilbert space $H$, such that we have the identification
\begin{eqnarray}
\label{eq:vectorValuedRangeSpace}
Z_p = L^p(((0,\infty),r^{n-1} dr),H).
\end{eqnarray}
\end{enumerate}
\end{definition}
Note that from a change of variables follows the identity
\begin{eqnarray*}
\label{eq:scalingNorm}
\forall p \in [1,\infty], \, \lambda > 0: \; \Vert f(\lambda \cdot) \Vert_{Z_p} = \lambda^{-\frac{n}{p}} \Vert f \Vert_{Z_p},
\end{eqnarray*}
which we shall refer to as the $L^p$-dilation property.\\
Also note that the independence of $R$ in \eqref{eq:compFrequLoc} follows from a scaling argument and that from \eqref{eq:vectorValuedRangeSpace} follows the duality relation 
\begin{eqnarray}
\label{eq:dualityRelation}
\left(Z_p \right)^\prime \simeq Z_{p^\prime}, Z_p \mbox{ separable for }p \in [1,\infty), \mbox{ where } \frac{1}{p}+\frac{1}{p^\prime}=1.
\end{eqnarray}
In our applications the $Z_p$-spaces are the $L^p(\mathbb{R}^n)$-spaces of spherically symmetric functions, when we consider spherical symmetry (i.e. $H=\mathbb{C}$ in \eqref{eq:vectorValuedRangeSpace}), or spaces of the kind $\mathcal{L}_r^p L_\omega^q$, when we consider weakened integrability in the spherical coordinates (i.e. $H= L_\omega^2$ in \eqref{eq:vectorValuedRangeSpace}). We make the following definition on adding derivatives:
\begin{equation*}
Z^s_p = \left\{ f \in \mathcal{S}^\prime(\mathbb{R}^n)/\mathcal{P} \, | \, \Vert f \Vert_{Z^s_p} < \infty \right\}
\end{equation*}
with the Besov-like norm
\begin{equation*}
\Vert f \Vert_{Z^s_p} = \left( \sum_N N^{2s} \Vert P_N f \Vert^2_{Z_p} \right)^{1/2}.
\end{equation*}
Working with these norms has the benefit, that one can conclude estimates from frequency-localized versions although one does not necessarily have a Littlewood-Paley decomposition of the considered spaces. 
 When we consider operators between the spaces, we shall frequently start with a frequency localized estimate.
We also consider the following norms
 \begin{equation*}
 \Vert F \Vert_{Z^s_{p,q}} = \left( \sum_N N^{2s} \Vert P_N F \Vert_{L_t^q Z_p}^2 \right)^{1/2}.
 \end{equation*}
Also, we use the nomenclature of an 'extended decay parameter' which is supposed to be understood morally: The decay estimates \eqref{eq:untruncatedDecayEstimate}, \eqref{eq:truncatedDecayEstimate} are not improved. But for instance in the case of spherically symmetric functions the sharp decay only holds for a relatively thin set (cf. \cite{Sterbenz2005}), which makes the proof of the generalized homogeneous Strichartz estimates possible and with the ordinary decay parameter giving rise to the classical sharp line of integrability coefficients (cf. \cite{TaoKeel1998}) we want this phenomenon to extend for generalized Strichartz estimates.
With range spaces, which have the compatibility property, we formulate the generalized homogeneous Strichartz estimates as follows. 
\begin{definition}[Generalized homogeneous Strichartz estimates]
Let $a \geq 1$ and $\mathcal{U}_a$ as above. We say that $\mathcal{U}_a$ admits generalized homogeneous Strichartz estimates with range spaces $(Z_p)_{p \in [1,\infty]}$, which are a family of spaces of  tempered distributions in $\mathbb{R}^n$ with the compatibility property, and with extended decay parameter $\sigma^\prime$\footnote{Again, $\sigma^\prime$ will typically depend on $a$ and $n$, but we suppress the dependence for the sake of brevity for the same reasons we suppress the dependence considering $\sigma$.}, if $1\leq\sigma<\sigma^\prime$\footnote{The requirement $\sigma \geq 1$ is a technical restriction to avoid endpoints with infinite space integrability. Typically, the estimates for these endpoints are ruled out and must be excluded by force.} and we find the estimate
\begin{equation}
\label{eq:genStrichartzEstimateFrequLoc}
\left\Vert P_N e^{it D^a} u_0 \right\Vert_{L_t^q Z_p} \lesssim_{n,p,q} N^{-s} \Vert u_0 \Vert_{L^2}
\end{equation}
to hold for any $N \in 2^{\mathbb{Z}}$ with\footnote{For $1/q \leq \sigma (1/2-1/p)$ we have the ordinary Strichartz estimates and typically, we do not have the generalized Strichartz estimates at some point for the critical decay parameter $\sigma^\prime$, or not all.}
\begin{equation}
\label{eq:genStrichartzCond}
\frac{1}{q} = \tau \left( \frac{1}{2} - \frac{1}{p} \right), \; \tau \in (\sigma, \sigma^\prime), \; s= - \frac{n}{2} + \frac{n}{p} + \frac{a}{q} , \; q,p \geq 2.
\end{equation}
\end{definition}
In the following we shall also refer to generalized homogeneous Strichartz estimates as generalized Strichartz estimates.\\
Furthermore, we note that by duality follows from \eqref{eq:genStrichartzEstimateFrequLoc}
\begin{equation*}
\left\Vert P_N \int_{-\infty}^{\infty} e^{-i \tau D^a} F(\tau) d\tau \right\Vert_{L^2} \lesssim_{n,\tilde{p},\tilde{q}} N^{-\tilde{s}} \Vert F \Vert_{L_t^{\tilde{q}^\prime} Z_{\tilde{p}^\prime}}
\end{equation*}
and taking the two estimates together yields
\begin{equation}
\label{eq:globalInhomogEst}
\left\Vert P_N \int_{-\infty}^{\infty} e^{i(t-s)D^a} F(s) ds \right\Vert_{L_t^{\tilde{q}} Z_{\tilde{p}}} \lesssim_{n,p,q,\tilde{p},\tilde{q}} N^{-(s+\tilde{s})} \Vert F \Vert_{L_t^{q^\prime} Z_{p^\prime}}.
\end{equation}
We work under the convention that estimates of the kind \eqref{eq:genStrichartzEstimateFrequLoc} and \ref{eq:globalInhomogEst} hold for any $N \in 2^{\mathbb{Z}}$ without further comment. Also note that in any of the three estimates we can of course assume the function under consideration to be frequency localized around $N$.\\

\section{Proof of the global inhomogeneous estimates}
\label{section:proof}
Starting from generalized Strichartz estimates, we shall follow the strategy already developed in \cite{Foschi2005}:
\begin{enumerate}
\item[1.] Finding temporally delayed and localized inhomogeneous estimates with normalized temporal support, which will be done in Section \ref{subsection:localInhomogeneousEstimates},
\item[2.] finding more local estimates by means of a scaling transform in Section \ref{subsection:scalingSymmetry},
\item[3.] perceiving global inhomogeneous estimates as bilinear estimates, which can be decomposed into local estimates by means of a Whitney decomposition,
\item[4.] summing the local estimates, which becomes possible through atomic decompositions of the involved functions. This will be done in Section \ref{subsection:globalInhomogeneousEstimates}.
\end{enumerate}
 \subsection{Local estimates}
 \label{subsection:localInhomogeneousEstimates}
 Let $I$ and $J$ be two time intervals of unit length $|I| = |J| = 1$, which are separated, so that $d=$ dist$(I,J)  \sim 1$ when $\sup I < \inf J$ and let us consider the local estimates
\begin{equation}
\label{eq:localInhomogEstII}
\left\Vert \int_{-\infty}^\infty e^{i(t-\tau)D^a} F(\tau) \, d\tau \right\Vert_{L_t^{\tilde{q}}(J;Z^{\tilde{s}}_{\tilde{p}})} \leq C_{n,p,q,\tilde{p},\tilde{q}} \Vert F \Vert_{L_t^{q^\prime}(I;Z^{-s}_{p^\prime})},
\end{equation}
which are meant to hold for any time intervals with the properties described above.\\
We note that such an estimate also depends on $|I|$, $|J|$ and $d$. For the moment we suppress the dependence, but we have to keep track of it, when we consider intervals of different shape, which will be done in the next sections.\\
Assuming $F$ to be supported in $I$, we find that $TT^*$ and $\left(TT^* \right)_R$ coincide.
The aim of this section is to find as many of these local estimates as possible. In the following lemma we observe that it is enough to consider frequency localized variants if $q,\tilde{q} \geq 2$ due to Minkowski's inequality.
\begin{lemma}
\label{lem:localInhomogEstII}
Suppose that the estimate 
\begin{equation*}
\label{eq:localinhomogest}
\left\Vert P_N \int_{-\infty}^t e^{i D^a (t-s)} F(s) ds \right\Vert_{L_t^{\tilde{q}}(J;Z_{\tilde{p}})} \leq C_{n,p,q,\tilde{p},\tilde{q}} N^{-(s+\tilde{s})} \Vert F \Vert_{L_t^{q^\prime}(I;Z_{p^\prime})}
\end{equation*}
holds for some $q,\tilde{q},p,\tilde{p},s$ and any intervals $I$ and $J$ of unit length, which are separated so that $d=$dist$(I,J)\sim 1$ and $\sup I < \inf J$ with $q,\tilde{q} \geq 2$.\\
Then we also find the estimate \eqref{eq:localInhomogEstII} to hold.
\end{lemma}
In order to find local estimates, we start from the dispersive estimate:
\begin{equation*}
\label{eq:unitFrequenciesDispersion}
\Vert P_1 e^{it D^a} u_0 \Vert_{L_x^{\infty}} \leq C_{n} |t|^{-\sigma} \Vert u_0 \Vert_{L_x^1} \; \; (t \neq 0).
\end{equation*}
By means of a scaling transform we find
\begin{equation*}
\label{eq:generalFrequenciesDispersion}
\Vert P_N e^{itD^a} u_0 \Vert_{L_x^\infty} \leq C_n |t|^{-\sigma} N^{n-a\sigma} \Vert u_0 \Vert_{L_x^1} \; \; (t \neq 0).
\end{equation*}
Integrating the above inequality yields the following local estimate:
\begin{equation}
\label{eq:localDispEst}
\left\Vert P_N \int_{-\infty}^t e^{i(t-s)D^a} F(s) ds \right\Vert_{L_t^\infty(J; L_x^\infty)} \leq C_{n} N^{n-a\sigma} \left\Vert F \right\Vert_{L_t^1(I; L_x^1)}
\end{equation}
In the following we set $$r= \frac{a\sigma-n}{2}.$$
We would like to combine \eqref{eq:localDispEst} with the estimates we find from \eqref{eq:globalInhomogEst}; however, we have to require that these estimates still hold in the $Z_p$-spaces, which leads us to the following definition:
\begin{definition}[Generalized dispersive estimate]
Suppose that $\mathcal{U}_a$ admits generalized Strichartz estimates with range spaces $(Z_p)_{p \in [1,\infty]}$ and extended decay parameter $\sigma^\prime$. We say that these generalized Strichartz estimates admit a generalized dispersive estimate, if we have
\begin{equation}
\label{eq:genDispEst}
\left\Vert P_N \int_{-\infty}^t e^{iD^a(t-\tau)} F(\tau) d\tau \right\Vert_{L_t^\infty(J; Z_\infty)} \leq C_{n} N^{n-a\sigma} \left\Vert F \right\Vert_{L_t^1(I; Z_1)},
\end{equation}
which is supposed to hold for any intervals $I$ and $J$ of unit length $|I| = |J| = 1$, which are separated so that dist$(I,J) \sim 1$.
\end{definition}
We will also make use of the interpolation identity
\begin{equation}
\label{eq:interpolationIdentity}
(Z_p, Z_q)_{[\theta]} = Z_u, \; \frac{1}{u} = \frac{1-\theta}{p} + \frac{\theta}{q}, \; q,p \in [1, \infty], \; \theta \in [0,1],
\end{equation}
but this identity follows for $p \neq \infty $, $q \neq \infty $ from the $Z_p$-spaces being vector-valued $L^p$-spaces; if one of the coefficients is infinite, we have to take into account the regularity and decay assumption, see again \cite[Theorem~5.1.2.,~p.~107]{Bergh1976}. The same reasoning holds for further nested $L^p$-spaces.\\
We observe that in \eqref{eq:localinhomogest} there is some ambiguity between $s$ and $\tilde{s}$ in the sense that the estimate actually only depends on $s+\tilde{s}$. Therefore, we will only consider estimates of the kind
\begin{equation}
\label{eq:localInhomogEstIII}
\left\Vert P_N \int_{-\infty}^t e^{iD^a(t-\tau)} F(\tau) d\tau \right\Vert_{L_t^{\tilde{q}}(J;Z_{\tilde{p}})} \lesssim_{n,p,q,\tilde{p},\tilde{q}} N^{-2s} \Vert F \Vert_{L_t^{q^\prime}(I;Z_{p^\prime})}
\end{equation}
and
\begin{equation}
\label{eq:localInhomogEstIV}
\left\Vert \int_{-\infty}^t e^{iD^a(t-\tau)} F(\tau) d\tau \right\Vert_{L_t^{\tilde{q}}(J;Z^{s}_{\tilde{p}})} \lesssim_{n,p,q,\tilde{p},\tilde{q}} \Vert F \Vert_{L_t^{q^\prime}(I;Z^{-s}_{p^\prime})}.
\end{equation}
When we perform interpolation steps to find global estimates in Section  \ref{subsection:globalInhomogeneousEstimates}, the above representations will be advantageous. We work out the local estimates below using the methods from \cite{Foschi2005}: That is interpolating the estimates from factorization with the dispersive estimate and finally using H\"older's inequality in time. The following theorem is an extension to \cite[Theorem~1.12.,~p.~4]{Foschi2005}; the proof is complicated from the necessity to keep in mind the derivative parameters.
\begin{theorem}[Local inhomogeneous estimates]
\label{thm:localEstimates}
Suppose that the family of linear operators $\mathcal{U}_a$ admits generalized homogeneous Strichartz estimates with range spaces $(Z_p)_{p \in [1,\infty]}$, with extended decay parameter $\sigma^\prime$ and which admit a generalized dispersive estimate.\\
Then we find the estimates \eqref{eq:localInhomogEstIII} and \eqref{eq:localInhomogEstIV} to hold if
\begin{align*}
\exists \; \sigma_1, \sigma_2 \in (\sigma,\sigma^\prime):\\
1 \leq &\mu = \frac{(a/2) \left( \sigma_1/2 + \sigma_2/2 - \sigma \right)}{s-r + \left( (a \sigma_1 -n)/p + (a \sigma_2 -n)/\tilde{p} \right)/2} < \infty \\
\frac{\sigma_1 - 1}{\sigma_1} \frac{p}{2} \leq \mu \leq \frac{p}{2} &\; \; \; \; \frac{\sigma_2 -1}{\sigma_2} \frac{\tilde{p}}{2} \leq \mu \leq \frac{\tilde{p}}{2} \nonumber \\
\frac{(\sigma_1/2)}{(1/q) + (\sigma_1/p)} \leq \mu &\; \; \; \; \frac{(\sigma_2/2)}{(1/\tilde{q}) + (\sigma_2/\tilde{p})} \leq \mu \nonumber
\end{align*}
\end{theorem}
\begin{proof}
Let us set $\mathcal{E}_{loc} = \left\{ (1/q,1/p,1/\tilde{q},1/\tilde{p},s) \in [0,1]^4 \times \mathbb{R} \, | \, \right. $ \eqref{eq:localInhomogEstIII} is valid $\left. \right\}$.
We have $(0,0,0,0,r) \in \mathcal{E}_{loc}$ by virtue of the dispersive estimate \eqref{eq:localDispEst}.
We also observe that if the generalized homogeneous Strichartz estimates are valid, that is $(Q,P,S)$ and $(\tilde{Q},\tilde{P},\tilde{S})$ satisfy \eqref{eq:genStrichartzCond}, then we find that $(1/Q,1/P,1/\tilde{Q},1/\tilde{P},(S+\tilde{S})/2) \in \mathcal{E}_{loc}$ due to \eqref{eq:globalInhomogEst} and localization in time.\\
That means we have for some $\sigma_1, \sigma_2 \in (\sigma, \sigma^\prime)$
\begin{equation}
\label{eq:localInhomogStartI}
\begin{split}
&\frac{1}{Q} = \sigma_1 \left( \frac{1}{2} - \frac{1}{P} \right), \; \; \; S = -n \left( \frac{1}{2} - \frac{1}{P} \right) + \frac{a}{Q}, \\
 &\frac{1}{\tilde{Q}} = \sigma_2 \left( \frac{1}{2} - \frac{1}{\tilde{P}} \right), \; \; \; \tilde{S} = -n \left( \frac{1}{2} - \frac{1}{\tilde{P}} \right) + \frac{a}{\tilde{Q}}, \; \; \; 0 \leq \frac{1}{Q}, \frac{1}{\tilde{Q}}, \frac{1}{P}, \frac{1}{\tilde{P}} \leq \frac{1}{2}.
 \end{split}
\end{equation}
%but if we take $\sigma_1=1$, we must not consider $Q=2,P=\infty$ and if we take $\sigma_2 = 1$, the point $\tilde{Q}=2,\tilde{P}=\infty$ is ruled out.
Note that $S$ and $\tilde{S}$ only depend on $\frac{1}{P}, \; \frac{1}{\tilde{P}}$, respectively, that is
\begin{equation*}
S= (a \sigma_1 - n) \left( \frac{1}{2} - \frac{1}{P} \right), \; \; \tilde{S} = (a \sigma_2 - n) \left( \frac{1}{2} - \frac{1}{\tilde{P}} \right).
\end{equation*}
Second, we can take the convex hull of these points with the point $(0,0,0,0,r)$ by virtue of the interpolation property. This gives rise to the wedge in Figure \ref{fig:globalInhomogeneousEstimates}. Those points are of the form $$(\theta/Q, \theta/P, \theta/\tilde{Q}, \theta/\tilde{P}, (1-\theta) r + \theta \frac{(S+\tilde{S})}{2}).$$
Before we apply H\"older's inequality to find more admissible coefficients with respect to time integrability, we lift the estimates of the kind \eqref{eq:localInhomogEstIII} to estimates of the kind \eqref{eq:localInhomogEstIV} by invoking Lemma \ref{lem:localInhomogEstII}, which is possible because $\frac{\theta}{Q}, \frac{\theta}{\tilde{Q}} \leq \frac{1}{2}$. Lifting is no longer possible after applying H\"older's inequality in general.\\
Finally, we apply H\"older's inequality to find points of the form 
\begin{equation}
\begin{split}
(1/q,1/p,1/\tilde{q},1/\tilde{p},s) \mbox{ where }
\frac{1}{q} \geq \frac{\theta}{Q}, \; \; \; \frac{1}{p} = \frac{\theta}{P}, \\
 \frac{1}{\tilde{q}} \geq \frac{\theta}{\tilde{Q}}, \; \; \; \frac{1}{\tilde{p}} = \frac{\theta}{\tilde{P}}, \; \; \; s=(1-\theta)r + \theta (S+\tilde{S})/2,
\end{split}
\end{equation}
when $(Q,P,S,\tilde{Q},\tilde{P},\tilde{S})$ satisfies \eqref{eq:localInhomogStartI}. Our aim is to eliminate the dependence from the initial values. We can already eliminate $P, \tilde{P}$:
\begin{align*}
\frac{\theta}{Q} = \sigma_1 \left( \frac{\theta}{2} - \frac{1}{p} \right)&\,&  \frac{\theta}{\tilde{Q}} = \sigma_2 \left( \frac{\theta}{2} - \frac{1}{\tilde{p}} \right) \\
0 \leq \frac{\theta}{Q}, \, \frac{\theta}{\tilde{Q}} \leq \frac{\theta}{2}&\,&  0 \leq \frac{1}{p}, \frac{1}{\tilde{p}} \leq \frac{\theta}{2}, \, \theta \in (0,1] \\
\frac{1}{q} \geq \frac{\theta}{Q}&\,& \, \frac{1}{\tilde{q}} \geq \frac{\theta}{\tilde{Q}} \\
s=r+\theta \left( \frac{S+\tilde{S}}{2} - r \right) &\,& \; \\
\theta S = (a \sigma_1 -n) \left( \frac{\theta}{2} - \frac{1}{p} \right) &\,& \theta \tilde{S} = (a \sigma_2 - n) \left( \frac{\theta}{2} - \frac{1}{\tilde{p}} \right)
\end{align*}
%The ruled out double-endpoint $(2,\infty)$ for a decay parameter $\sigma=1$ translates to $p < \infty$ or $\tilde{p} < \infty$, respectively.
Next, we can eliminate $Q, \tilde{Q}$ and find $\theta$ from the last equations:
\begin{align*}
\sigma_1 \left( \frac{\theta}{2} - \frac{1}{p} \right) \leq \frac{\theta}{2} \; &\;& \; \sigma_2 \left( \frac{\theta}{2} - \frac{1}{\tilde{p}} \right) \leq \frac{\theta}{2} \\
0 \leq \frac{1}{p}, \frac{1}{\tilde{p}} \leq \frac{\theta}{2} \; &\;& \; 0 < \theta \leq 1 \\
\frac{1}{q} \geq \sigma_1 \left( \frac{\theta}{2} - \frac{1}{p} \right) \; &\;& \;
\frac{1}{\tilde{q}} \geq \sigma_2 \left( \frac{\theta}{2} - \frac{1}{\tilde{p}} \right) \\
\theta = \frac{s-r + \left( (a \sigma_1 -n)/p + (a \sigma_2 -n)/\tilde{p} \right)/2}{(a/2) \left( \sigma_1/2 + \sigma_2/2 - \sigma \right)}
%\frac{s}{a} - \frac{\theta \alpha}{a} - \frac{n}{r a} = \frac{\theta}{Q} \leq \frac{1}{q}, \; \; \; \frac{\tilde{s}}{a} - \frac{\theta \alpha}{a} - \frac{n}{\tilde{r} a} = \frac{\theta}{\tilde{Q}} \leq \frac{1}{\tilde{q}}
\end{align*}
We rearrange these inequalities:
\begin{align*}
1 \leq \frac{1}{\theta} < \infty &\,&\\
\frac{\sigma_1 - 1}{\sigma_1} \frac{\theta}{2} \leq \frac{1}{p} \leq \frac{\theta}{2} &\;& \; 
\frac{\sigma_2 - 1}{\sigma_2} \frac{\theta}{2} \leq \frac{1}{\tilde{p}} \leq \frac{\theta}{2} \\
\frac{\sigma_1 \theta}{2} \leq \frac{1}{q} + \frac{\sigma_1}{p} &\;& \; \frac{\sigma_2 \theta}{2} \leq \frac{1}{\tilde{q}} + \frac{\sigma_2}{\tilde{p}} \\
\frac{s-r + \left( (a \sigma_1 -n)/p + (a \sigma_2 -n)/\tilde{p} \right)/2}{(a/2) \left( \sigma_1/2 + \sigma_2/2 - \sigma \right)} &=\theta& \;
%\frac{s}{a} - \frac{n}{r a} - \frac{1}{q} \leq \frac{\theta -\frac{n}{2}}{a}, \; \; \frac{\tilde{s}}{a} - \frac{n}{\tilde{r} a} - \frac{1}{\tilde{q}} \leq \frac{\theta -\frac{n}{2}}{a}
\end{align*}
Next, we isolate the quantity $1/\theta$:
\begin{align}
1 \leq \frac{1}{\theta} < \infty &\;& \label{eq:isolatedTheta} \\
\frac{\sigma_1 - 1}{\sigma_1} \frac{p}{2} \leq \frac{1}{\theta} \leq \frac{p}{2} &\;& \frac{\sigma_2 -1}{\sigma_2} \frac{\tilde{p}}{2} \leq \frac{1}{\theta} \leq \frac{\tilde{p}}{2} \nonumber \\
\frac{(\sigma_1/2)}{(1/q) + (\sigma_1/p)} \leq \frac{1}{\theta} &\;& \frac{(\sigma_2/2)}{(1/\tilde{q}) + (\sigma_2/\tilde{p})} \leq \frac{1}{\theta} \nonumber \\
\frac{(a/2) \left( \sigma_1/2 + \sigma_2/2 - \sigma \right)}{s-r + (1/2) \left( (a \sigma_1 -n)/p + (a \sigma_2 -n)/\tilde{p} \right)} &=\frac{1}{\theta}& \; \nonumber
\end{align}
To find the conditions on the derivative parameter, we plug in the value we found for $1/\theta=\mu$ in terms of the derivative parameter into \eqref{eq:isolatedTheta} which finishes the proof.
\end{proof}
\subsection{Scaling symmetry}
\label{subsection:scalingSymmetry}
Next, we apply the scaling symmetry already mentioned above to find rescaled local estimates, which become useful when we recover the global estimates.\\
The following proposition is an extension to {\cite[Proposition~2.1.,~p.~6]{Foschi2005}}.
\begin{lemma}[Rescaled local estimates]
\label{lem:scalingLemma}
Suppose that the family $\mathcal{U}_a$ admits the estimate \eqref{eq:localInhomogEstIII} or \eqref{eq:localInhomogEstIV} for some $q,p,\tilde{q},\tilde{p},s$ with $p,\tilde{p} \in (1,\infty)$\footnote{This is a technical restriction, which does not have any practical relevance for us because in Theorem \ref{thm:localEstimates} we found $2\leq p,\tilde{p} < \infty$.} and any two time intervals of unit length $|I| = |J| = 1$, which are separated so that dist$(I,J)\sim 1$, when $\sup I < \inf J$ and that the $Z_p$-spaces have the compatibility property.\\
Letting $\tilde{I}, \tilde{J}$ denote two time intervals with $|\tilde{I}| = |\tilde{J}| = \lambda$, which are separated so that dist$(\tilde{I},\tilde{J})\sim \lambda$, when $\sup \tilde{I} < \inf \tilde{J}$, in case of \eqref{eq:localInhomogEstIII} we find the estimate
\begin{equation}
\label{eq:rescaledInhomogEstI}
\left\Vert P_N \int_{-\infty}^t e^{i(t-s)D^a} F(s) ds \right\Vert_{L_t^{\tilde{q}}(\tilde{J}, Z^{s}_{\tilde{p}})} \lesssim_{n,p,q,\tilde{p},\tilde{q}} \lambda^{\beta_a(q,\tilde{q},p,\tilde{p},s)} N^{-2s} \Vert F \Vert_{L_t^{q^\prime}(\tilde{I}; Z^{-s}_{p^\prime})} 
\end{equation}
to be true and in case of \eqref{eq:localInhomogEstIV} we find the following estimate to hold
\begin{equation}
\left\Vert P_N \int_{-\infty}^t e^{i(t-s)D^a} F(s) ds \right\Vert_{L_t^{\tilde{q}}(\tilde{J}, Z^{s}_{\tilde{p}})} \lesssim_{n,p,q,\tilde{p},\tilde{q}} \lambda^{\beta_a(q,\tilde{q},p,\tilde{p},s)} \Vert F \Vert_{L_t^{q^\prime}(\tilde{I}; Z^{-s}_{p^\prime})}.
\label{eq:rescaledInhomogEstII} 
\end{equation}
\end{lemma}
\begin{proof}
The proof consists only of straight-forward changes of variables.
\end{proof}
\subsection{Recovering the global estimates}
\label{subsection:globalInhomogeneousEstimates}
We follow the strategy from \cite{Foschi2005}, that is considering a Whitney decomposition of the domain of integration, so that the local estimates come into play and gaining summability by perturbation of the coefficients. The execution is complicated from the additional derivative parameter.\\
To prove the inhomogeneous estimates \eqref{eq:globalEstimateI} and \eqref{eq:globalEstimateII}, we adapt the bilinear formulation of $\left(TT^*\right)_{R}$ by setting
\begin{equation*}
\begin{split}
 B: L_t^{q^\prime} Z^{-s}_{p^\prime} \times L_t^{\tilde{q}^\prime} Z^{-\tilde{s}}_{\tilde{p}^\prime} &\rightarrow \mathbb{C} \\
 B(F,G) &= \int \int_{s<t} \, \langle  U(-s)  F(s), \, U(-t) G(t) \rangle \, ds \, dt
 \end{split}
 \end{equation*}
 and we also have to consider the following variant with frequency localization
 \begin{equation*}
 \begin{split}
 B^N: L_t^{q^\prime} Z_{p^\prime} \times L_t^{\tilde{q}^\prime} Z_{\tilde{p}^\prime} &\rightarrow \mathbb{C} \\
 B^N(F,G) &= \int \int_{s<t} \, \langle P_N U(-s)  F(s), \, U(-t) G(t) \rangle \, ds \, dt
 \end{split} 
 \end{equation*}
 Note that we have to keep track of the derivative parameters when estimating the latter form.\\
 We can exploit our previous results on local estimates by considering Whitney's dyadic decomposition of the domain of integration $\Omega = \left\{ (s,t) \in \mathbb{R}^2 \, | \, s < t \right\}$. Recall that a dyadic cube in Euclidean space is a cube whose sidelength is a dyadic number $\lambda \in 2^{\mathbb{Z}}$ and the coordinates of its vertices are integer multiples of $\lambda$. Precisely, we use the  theorem \cite[Appendix~J,~pp.~463-464]{Grafakos2008} on decomposition of open sets in Euclidean space into essentially disjoint dyadic cubes.
 By $\mathcal{Q}$ we denote the Whitney decomposition of $\Omega$.\\
 For each dyadic number $\lambda$, by $\mathcal{Q}_{\lambda}$ we denote the collection of squares in $\mathcal{Q}$ with sidelength $\lambda$. Each square $Q=I \times J \in \mathcal{Q}_{\lambda}$ satisfies the condition
 \begin{equation*}
 \lambda = |I| = |J| \sim dist(Q, \partial \Omega) \sim dist(I,J).
 \end{equation*}
 Transferring the Whitney decomposition onto the bilinear form we arrive at
 \begin{equation}
 \label{eq:bilinearSumWhitneyII}
 B = \sum_{\lambda} \sum_{Q \in \mathcal{Q}_{\lambda}} B_Q,
 \end{equation}
 where we have restricted the domain of integration to $Q=I \times J$ on $B_Q$, that is
 \begin{equation*}
 \label{eq:bilinearSumWhitneyI}
 B_Q(F,G) = B(\chi_I F, \chi_J G) = \int \int_{s \in I,\; t \in J} \langle U(-s) F(s), \, U(-t) G(t) \rangle \, ds \, dt.
 \end{equation*}
 Note that when we consider estimates for $B_Q, \; Q=I \times J$, $F$ and $G$ are effectively supported on $I$ and $J$, respectively: Thus, we find that the estimate \eqref{eq:rescaledInhomogEstII} is equivalent to 
 \begin{equation}
 \label{eq:bilinearLocalInhomogEst}
 |B_Q(F,G)| \lesssim_{n,p,q,\tilde{p},\tilde{q}} \lambda^{\beta_a(q,\tilde{q},p,\tilde{p},s)} \Vert F \Vert_{L_t^{\tilde{q}^\prime}(I;Z^{-s}_{\tilde{p}^\prime})} \Vert G \Vert_{L_t^{q^\prime}(J;Z^{-s}_{p^\prime})}, \; \; Q \in \mathcal{Q}_\lambda.
 \end{equation}
 Note that the above arguments also apply to $B^N$ with the obvious modifications.\\
 To recover the global estimates from the local ones, we have to perform the summations. First, we note the following variant to H\"older's inequality for sequence spaces by combining ordinary H\"older's inequality with the embedding $\ell^p \hookrightarrow \ell^q$ for $p \leq q$:
 \begin{lemma}[{\cite[Lemma~3.2.,~p.~8]{Foschi2005}}]
 \label{lem:HoelderSequences}
 Suppose that $1/r + 1/\tilde{r} \geq 1$, then we have
 \begin{equation*}
 \sum_{\begin{array}{l} Q \in \mathcal{Q}_{\lambda} \\ Q=I \times J \end{array}} \Vert f \Vert_{L^{\tilde{r}}(I)} \Vert g \Vert_{L^r(J)} \leq \Vert f \Vert_{L^{\tilde{r}}(\mathbb{R})} \Vert g \Vert_{L^r(\mathbb{R})}
 \end{equation*}
 for $f \in L^{\tilde{r}}(\mathbb{R})$, $g \in L^r(\mathbb{R})$ and any dyadic number $\lambda$.
 \end{lemma}
 An application of Lemma \ref{lem:HoelderSequences} to \eqref{eq:bilinearLocalInhomogEst} under the assumption $1/q + 1/\tilde{q} \leq 1$ yields
 \begin{equation}
 \label{eq:bilinearEstimateSummation}
 \sum_{Q \in \mathcal{Q}_\lambda} |B_Q(F,G)| \lesssim_{n,p,q,\tilde{p},\tilde{q}} \lambda^{\beta_a(q,\tilde{q},p,\tilde{p},s)} \Vert F \Vert_{L_t^{q^\prime}(\mathbb{R};Z^{-s}_{p^\prime})} \Vert G \Vert_{L_t^{\tilde{q}^\prime}(\mathbb{R};Z^{-s}_{\tilde{p}^\prime})}.
 \end{equation}
 Since the operator $\left( TT^*\right)_R$ has a convolution structure, we do not lose any globally admissible pairs making this assumption.\\
 We will also need the following variant of Young's inequality:
 \begin{lemma}[{\cite[Lemma~4.3.,~p.~11]{Foschi2005}}]
 \label{lem:YoungSequences}
 Let $\left( A_n \right), \left( B_n \right), \left( C_n \right)$ be sequences of non-negative numbers. If
 \begin{equation*}
 \frac{1}{p} + \frac{1}{q} + \frac{1}{r} \geq 2,
 \end{equation*}
 then $\sum_{n,k} A_n B_k C_{n-k} \leq \Vert A \Vert_{\ell^p} \Vert B \Vert_{\ell^q} \Vert C \Vert_{\ell^r}$.
 \end{lemma}
 Apparently, necessary for performing the summations in \eqref{eq:bilinearSumWhitneyII} is the scaling condition $\beta_a = 0$. But still, plain summation does not work; however, perturbation of the exponents enhances summability because of the exponential decay one has in a neighbourhood as was well demonstrated in \cite{TaoKeel1998} and \cite{Foschi2005}.\\[0.5cm]
 First, we consider the non-sharp case $$1/q + 1/\tilde{q} < 1,$$ which allows us to perturb the time integrability exponents. We prepare this interpolation step by introducing atomic decompositions of $L^p$-functions. We resort to a vector-valued atomic decomposition of the functions $F \in L_t^{q^\prime} Z^{-s}_{p^\prime}$, $G \in L_t^{\tilde{q}^\prime} Z^{-\tilde{s}}_{\tilde{p}^\prime}$: The following part follows the steps from \cite[pp.~8-11]{Foschi2005} with the addition that in the concrete part of the proof, we are dealing with the unspecified Banach spaces $Z^{s}_p$, whereat in \cite{Foschi2005} the $L^p_X$-spaces were considered.
 \begin{definition}[$p$-atoms]
 Let $\mathcal{X}$ denote a measure space, $\mathcal{D}$ denote a Banach space and $1 \leq p \leq \infty$. A $p$-atom in $L^p(\mathcal{X};\mathcal{D})$ of size $\lambda$ is a measurable function $\varphi: \mathcal{X} \rightarrow \mathcal{D}$ such that:
 \begin{enumerate}
 \item[(i)] $\xi \mapsto \varphi(\xi)$ is supported on a set of measure less than $\lambda$;
 \item[(ii)] $\Vert \varphi \Vert_{L^\infty(\mathcal{X};\mathcal{D})} \lesssim_p \lambda^{-1/p}$.
 \end{enumerate}
 \end{definition}
 We note that $\Vert \varphi \Vert_{L^q(\mathcal{X};\mathcal{D})} \lesssim_p \lambda^{\frac{1}{q}-\frac{1}{p}}$, which will become mostly important, and we have the following result on atomic decompositions:
 \begin{lemma}[Atomic decomposition of $L^p(\mathcal{X};\mathcal{D})$-spaces, {\cite[Lemma~3.4,~p.~9]{Foschi2005}}]
 Any $\mathcal{D}$-valued function $F\in L^p(\mathcal{X};\mathcal{D})$ can be decomposed as
 \begin{equation*}
 F = \sum_{\lambda \in 2^{\mathbb{Z}}} a_\lambda \varphi_\lambda,
 \end{equation*}
 where:
 \begin{enumerate}
 \item[(a)] each $\varphi_\lambda$ is a $p$-atom in $L^p(\mathcal{X};\mathcal{D})$ of size $\lambda$;
 \item[(b)] the atoms $\varphi_\lambda$ have disjoint supports;
 \item[(c)] $a_\lambda$ are non-negative constants such that $\Vert f \Vert_{L^p(\mathcal{X};\mathcal{D})} \sim_p \Vert a_\lambda \Vert_{\ell^p(2^{\mathbb{Z}})}$.
 \end{enumerate}
 \end{lemma}
 Atomic decomposition of $F \in L_t^{q^\prime} Z^{-s}_{p^\prime}$ and $G \in L_t^{\tilde{q}^\prime} Z^{-\tilde{s}}_{\tilde{p}^\prime}$ with respect to the $L_t^q$-spaces yields
 \begin{equation*}
 F(t) = \sum_\mu a_\mu \varphi_\mu(t), \; \; G(t) = \sum_{\nu} b_\nu \psi_\nu(t),
 \end{equation*}
 where $\varphi_\mu$ is a $q^\prime$-atom with values in $Z^{-s}_{p^\prime}$ of size $\mu$ and 
 $\psi_\nu$ is a $\tilde{q}^\prime$-atom with values in $Z^{-\tilde{s}}_{\tilde{p}^\prime}$ of size $\nu$ and
 \begin{equation*}
 \Vert F \Vert_{L_t^{q^\prime} Z^{-s}_{p^\prime}} \sim_{q^\prime} \Vert a_\mu \Vert_{\ell^{q^\prime}}, \; \; \Vert G \Vert_{L_t^{\tilde{q}^\prime} Z^{-\tilde{s}}_{\tilde{p}^\prime}} \sim_{\tilde{q}^\prime} \Vert b_\nu \Vert_{\ell^{\tilde{q}^\prime}}.
 \end{equation*}
 Plugging the atomic decomposition into \eqref{eq:bilinearSumWhitneyII} we arrive at
 \begin{equation}
 \label{eq:bilinearSumWhitneyIII}
 B(F,G) = \sum_{\lambda, \mu, \nu} a_\mu b_\nu \sum_{Q \in \mathcal{Q}_\lambda} B_Q(\varphi_\mu,\psi_\nu).
 \end{equation}
 For convenience we introduce the notation $$[\lambda]  = \max\left\{ \lambda, \frac{1}{\lambda} \right\},$$ which will play the role of an absolute value for dyadic numbers in the following.\\
 The freedom to perturb the exponents gives rise to the following lemma, this is an extension to \cite[Lemma~4.2.,~p.~10]{Foschi2005}: 
 \begin{lemma}
 \label{lem:timewisePerturbation}
 Suppose that $1/q_0 + 1/\tilde{q}_0 < 1$, and that we have the local estimates \eqref{eq:localInhomogEstIII} and \eqref{eq:localInhomogEstIV} with exponents $(q,p,\tilde{q},\tilde{p},s)$ for all $(1/q,1/\tilde{q})$ in a full neighbourhood of $(1/q_0,1/\tilde{q}_0)$. Then, there exists $\varepsilon=\varepsilon(q_0,\tilde{q}_0) > 0$ such that, for all dyadic numbers $\lambda, \mu, \nu$, we have
 \begin{equation}
 \label{eq:timewisePerturbationI}
 \sum_{Q \in \mathcal{Q}_\lambda} | B_Q(\varphi_\mu,\psi_\nu)| \lesssim_{n,p,q_0,\tilde{p},\tilde{q}_0} \lambda^{\beta_a(q_0,\tilde{q}_0,p,\tilde{p},s)} \left[ \frac{\mu}{\lambda} \right]^{-\varepsilon} \left[ \frac{\nu}{\lambda} \right]^{-\varepsilon},
 \end{equation}
 whenever $\varphi_\mu$ is a $q_0^\prime$-atom of size $\mu$ in $L_t^{q_0^\prime} Z^{-s}_{p^\prime}$, and $\psi_\nu$ is a $\tilde{q}_0^\prime$-atom of size $\nu$ in $L_t^{\tilde{q}_0^\prime} Z^{-s}_{\tilde{p}^\prime}$ and
  \begin{equation}
  \label{eq:timewisePerturbationII}
 \sum_{Q \in \mathcal{Q}_\lambda} | B_Q^N(\varphi_\mu,\psi_\nu)| \lesssim_{n,p,q_0,\tilde{p},\tilde{q}_0} \lambda^{\beta_a(q_0,\tilde{q}_0,p,\tilde{p},s)} N^{-2s} \left[ \frac{\mu}{\lambda} \right]^{-\varepsilon} \left[ \frac{\nu}{\lambda} \right]^{-\varepsilon},
 \end{equation}
 whenever $\varphi_\mu$ is a $q_0^\prime$-atom of size $\mu$ in $L_t^{q_0^\prime} Z_{p^\prime}$, and $\psi_\nu$ is a $\tilde{q}_0^\prime$-atom of size $\nu$ in $L_t^{\tilde{q}_0^\prime} Z_{\tilde{p}^\prime}$.
 \end{lemma}
 \begin{proof}
 We can transfer the proof from \cite[Lemma~4.2.,~p.~10]{Foschi2005} because the proof only depends on the concrete form of the $q,\tilde{q}$-part of the scaling function, which coincide, and the properties of atomic decompositions, but we have to keep track of the derivative parameters.
 \end{proof}

We are ready to prove the first part of Theorem \ref{thm:globalInhomogeneousEstimates}:
 \begin{proof}[Proof of the non-sharp cases from Theorem \ref{thm:globalInhomogeneousEstimates}]
 We can imitate the proof from \cite[pp.~10-11]{Foschi2005} when we have to work at fixed derivative parameter in addition: Observe that we are in the position to employ Lemma \ref{lem:timewisePerturbation} because we made the inequalities, which involve the parameters $q$ and $\tilde{q}$, strict, and we required $1<q, \tilde{q} < \infty$ and plugging the estimates for atomic decomposition with respect to the $L_t^q$-spaces into \eqref{eq:bilinearSumWhitneyIII} we find
  \begin{equation*}
  | B(F,G) | \leq C_{n,p,q,\tilde{p},\tilde{q}} \sum_{\mu, \nu} a_\mu b_\nu \sum_\lambda \lambda^{\beta_a(q,\tilde{q},p,\tilde{p},s)} \left[ \frac{\mu}{\lambda} \right]^{-\varepsilon} \left[ \frac{\nu}{\lambda} \right]^{-\varepsilon}.
  \end{equation*}
  With the scaling condition $\beta_a = 0$, we can perform the sum over $\lambda$ and find
  \begin{equation*}
  \sum_\lambda \left[ \frac{\mu}{\lambda} \right]^{-\varepsilon} \left[ \frac{\nu}{\lambda} \right]^{-\varepsilon} \lesssim_\varepsilon \left( 1 + \log\left[ \frac{\mu}{\nu} \right] \right) \left[ \frac{\mu}{\nu} \right]^{-\varepsilon} = c_{\mu/\nu}.
  \end{equation*}
  Recall that $\varepsilon = \varepsilon(q,\tilde{q})$.
  Since the sequence $(c_\alpha)$ is absolutely summable, we can apply Lemma  \ref{lem:YoungSequences} on the estimate
  \begin{equation*}
  |B(F,G)| \leq C_{n,p,q,\tilde{p},\tilde{q}} \sum_{\mu,\nu} a_\mu b_\nu c_{\mu/\nu},
  \end{equation*}
  which is possible because $(a_\mu) \in \ell^{q^\prime}, \, (b_\nu) \in \ell^{\tilde{q}^\prime}, \, (c_\alpha) \in \ell^1$ and therefore $ 1/q^\prime + 1/\tilde{q}^\prime + 1 = 3- (1/q +1/\tilde{q}) > 2$ by hypothesis. This proves the estimate \eqref{eq:globalEstimateI}.\\
  For the second claim we consider the bilinear form $B^N$ and following along the above lines with the obvious modifications we arrive at the estimate
  \begin{equation*}
  \label{eq:frequLocEstNonSharp}
  \left\Vert P_N \int_{-\infty}^t U_a(t-s) F(s) ds \right\Vert_{L_t^{\tilde{q}} Z_{\tilde{p}}} \lesssim_{n,p,q,\tilde{p},\tilde{q}} N^{-2s} \Vert \tilde{P}_N F \Vert_{L_t^{q^\prime} Z_{p^\prime}}
  \end{equation*}
  and the estimate \eqref{eq:globalEstimateII} follows by squaring and summing over $N$.
 \end{proof}
For the sharp case $$ 1/q + 1/\tilde{q} = 1,$$ was considered a perturbation of the spatial exponents in \cite{Foschi2005}. Recall that we work under the assumption, that the $Z_p$-spaces are $L^p$-spaces, possibly vector-valued.\\
 Let $F(t) \in Z_{p^\prime}$ and let $G(t) \in Z_{\tilde{p}^\prime}$, which yields the atomic decompositions
\begin{equation*}
F(t) = \sum_{\mu} a_\mu(t) \varphi_\mu(t), \; \; \; G(t) = \sum_\nu b_\nu(t) \psi_\nu(t),
\end{equation*}
where $\varphi_\mu(t)$ is a $Z_{p^\prime}$-atom in $Z_{p^\prime}$ of size $\mu$ and 
$\psi_\nu(t)$ is a $Z_{\tilde{p}^\prime}$-atom in $Z_{\tilde{p}^\prime}$ of size $\nu$ and
\begin{equation*}
\Vert F(t) \Vert_{Z_{p^\prime}} \sim_{p^\prime} \Vert a_\mu(t) \Vert_{\ell^{p^\prime}}, \; \; \;
\Vert G(t) \Vert_{Z_{\tilde{p}^\prime}} \sim_{\tilde{p}^\prime} \Vert b_\nu(t) \Vert_{\ell^{\tilde{p}^\prime}}.
\end{equation*}
Plugging these decompositions into the bilinear form $B^N$ we arrive at
\begin{equation*}
B^N(F,G) = \sum_{\lambda, \mu, \nu} \sum_{Q \in \mathcal{Q}_\lambda} B^N_Q(a_\mu \varphi_\mu,b_\nu \psi_\nu).
\end{equation*}
We have the following lemma on enhanced summability due to perturbation of the spatial exponents at fixed derivative parameter $s$, which is found after \cite[Lemma~5.1.,~p.~12]{Foschi2005}:
 \begin{lemma}
 \label{lem:spatialPerturbation}
 Suppose that the local estimates \eqref{eq:localInhomogEstIII} hold with exponents $(q,p,\tilde{q},\tilde{p},s)$ for all $(1/p,1/\tilde{p})$ in a full neighbourhood of $(1/p_0,1/\tilde{p}_0)$. Then, there exists $\varepsilon = \varepsilon(p_0,\tilde{p}_0) > 0$, such that we have for all dyadic numbers $\lambda, \mu, \nu$ and the dyadic square $Q= I \times J \in \mathcal{Q}_\lambda$
 \begin{equation*}
 \begin{split}
 |B^N_Q(a \varphi_\mu, b \psi_\nu)| \lesssim_{n,p_0,q,\tilde{p}_0,\tilde{q}}& \lambda^{\beta_a(q,\tilde{q},p_0,\tilde{p}_0,s)} \Vert a \Vert_{L^{q^\prime}(I)} \Vert b \Vert_{L^{\tilde{q}^\prime}(J)} N^{-2s} \\
 & \left[ \frac{\mu}{\lambda^{n/a}} \right]^{- \varepsilon} \left[ \frac{\nu}{\lambda^{n/a}} \right]^{- \varepsilon}, 
 \end{split}
 \end{equation*}
 whenever $a \in L^{q^\prime}(I;\mathbb{R}), \; b \in L^{\tilde{q}^\prime}(J;\mathbb{R})$, and for each $t$, the function $\varphi_\mu(t)$ is a $p_0^\prime$-atom in $Z_{p_0^\prime}$ and the function $\psi_\nu(t)$ is a $\tilde{p}_0^\prime$-atom in $Z_{{\tilde{p}_0}^\prime}$ of size $\nu$.
 \end{lemma}
 \begin{proof}
 We employ the local estimates from \eqref{eq:localInhomogEstII} together with H\"older's inequality to find
 \begin{equation*}
 \begin{split}
 |B^N_Q(a \varphi_\mu,b \psi_\nu)| \leq& \; C_{n,p_0,q,\tilde{p}_0,\tilde{q}} \lambda^{\beta_a(q,\tilde{q},p,\tilde{p},s)}  \Vert a \Vert_{L^{q^\prime}(I)} \Vert b \Vert_{L^{\tilde{q}^\prime}(J)}  N^{-2s} \\
 & \Vert \varphi_\mu \Vert_{L_t^\infty(I;Z_{p^\prime})}  \Vert \psi_\nu \Vert_{L_t^\infty(J;Z_{\tilde{p}^\prime})} \\ 
 \leq& \; C_{n,p_0,q,\tilde{p}_0,\tilde{q}} \lambda^{\beta_a(q,\tilde{q},p,\tilde{p},s)} \Vert a \Vert_{L^{q^\prime}(I)} \Vert b \Vert_{L^{\tilde{q}^\prime}(J)} N^{-2s} \mu^{\frac{1}{p_0} - \frac{1}{p}} \nu^{\frac{1}{\tilde{p}_0}-\frac{1}{\tilde{p}}} \\
 =& \; C_{n,p_0,q,\tilde{p}_0,\tilde{q}} \lambda^{\beta_a(q,\tilde{q},p_0,\tilde{p}_0,s)} \Vert a \Vert_{L^{q^\prime}(I)} \Vert b \Vert_{L^{\tilde{q}^\prime}(J)} N^{-2s} \\
 & \; \left( \frac{\mu}{\lambda^{n/a}} \right)^{\frac{1}{p_0} - \frac{1}{p}} \left( \frac{\nu}{\lambda^{n/a}} \right)^{\frac{1}{\tilde{p}_0}-\frac{1}{\tilde{p}}}. 
 \end{split}
 \end{equation*}
 Since the local estimates hold in a full neighbourhood, we can choose $p$ and $\tilde{p}$ for given $\lambda, \mu, \nu$ with $\varepsilon = \varepsilon(p_0,\tilde{p}_0)>0$ and take the constant $C= C(n,p_0,q,\tilde{p}_0,\tilde{q})$ like in the proof of Lemma \ref{lem:timewisePerturbation}, so that
 \begin{equation*}
 \left( \frac{\mu}{\lambda^{n/a}} \right)^{\frac{1}{p_0} - \frac{1}{p}} = \left[ \frac{\mu}{\lambda^{n/a}} \right]^{- \varepsilon}, \; \; \; \left( \frac{\nu}{\lambda^{n/a}} \right)^{\frac{1}{\tilde{p}_0} - \frac{1}{\tilde{p}}} = \left[ \frac{\nu}{\lambda^{n/a}} \right]^{- \varepsilon},
 \end{equation*}
 which finishes the proof.
 \end{proof}
 We will perturb the spatial exponents $p$ and $\tilde{p}$ associated to decay parameters $\sigma_1$ and $\sigma_2$, but note that it is necessary to fix the sum of the derivative parameters. By this method we are unable to recover the estimates in the $L_t^q Z^s_p$-spaces, but only in the $Z^s_{p,q}$-spaces. In the specific case of spherical symmetry the $Z^s_p$-spaces become Besov spaces of spherically symmetric functions and the estimate can be concluded by an abstract interpolation argument via additional perturbation of the derivative parameters as demonstrated in \cite{Ovcharov2012}.
\begin{proof}[Proof of the sharp cases from Theorem \ref{thm:globalInhomogeneousEstimates}]
The above requirements on $(q,p,\tilde{q},\tilde{p},s)$ are sufficient to employ Lemma \ref{lem:spatialPerturbation} and we find
\begin{equation*}
\begin{split}
|B^N(F,G)| \leq& \; C_{n,p,q,\tilde{p},\tilde{q}} N^{-2s} \sum_{\lambda, \mu, \nu} \lambda^{\beta_a(q,\tilde{q},p,\tilde{p},s)} \left[ \frac{\mu}{\lambda^{n/a}} \right]^{- \varepsilon} \\
& \left[ \frac{\nu}{\lambda^{n/a}} \right]^{- \varepsilon} \sum_{I \times J \in \mathcal{Q}_\lambda} \Vert a_\mu \Vert_{L^{q^\prime}(I)} \Vert b_\nu \Vert_{L^{\tilde{q}^\prime}(J)}.
\end{split}
\end{equation*}
An application of Lemma \ref{lem:HoelderSequences}, which is possible due to $1/q + 1/\tilde{q} = 1 $, yields
\begin{equation*}
\begin{split}
|B^N(F,G)| \leq& \; C_{n,p,q,\tilde{p},\tilde{q}} N^{-2s} \sum_{ \mu, \nu} \Vert a_\mu \Vert_{L^{q^\prime}(\mathbb{R})} \Vert b_\nu \Vert_{L^{\tilde{q}^\prime}(\mathbb{R})} \\
& \left( \sum_{\lambda} \lambda^{\beta_a(q,\tilde{q},p,\tilde{p},s)} \left[ \frac{\mu}{\lambda^{n/a}} \right]^{- \varepsilon} \left[ \frac{\nu}{\lambda^{n/a}} \right]^{- \varepsilon} \right).
\end{split}
\end{equation*}
	Since we are at the scaling invariant case $\beta_a(q,\tilde{q},p,\tilde{p},s)=0$, we can perform the sum over $\lambda$ and find
\begin{equation*}
\sum_\lambda \left[ \frac{\mu}{\lambda^{n/a}} \right]^{- \varepsilon} \left[ \frac{\nu}{\lambda^{n/a}} \right]^{- \varepsilon} \lesssim_\varepsilon \left( 1 + \log \left[ \frac{\mu}{\nu} \right] \right) \left[ \frac{\mu}{\nu} \right]^{- \varepsilon} = c_{\mu/\nu}.
\end{equation*}
Therefore, we find
\begin{equation*}
|B^N(F,G)| \leq C_{n,p,q,\tilde{p},\tilde{q}} N^{-2s} \sum_{\mu,\nu} \Vert a_\mu \Vert_{L^{q^\prime}(\mathbb{R})} \Vert b_\nu \Vert_{L^{\tilde{q}^\prime}(\mathbb{R})} c_{\mu/\nu}.
\end{equation*}
As in the first part of the proof of Theorem \ref{thm:globalInhomogeneousEstimates}, the sequence $(c_\alpha)$ is absolutely summable and we can apply Lemma \ref{lem:YoungSequences} to arrive at the estimate
\begin{equation*}
\begin{split}
|B^N(F,G)| &\leq C_{n,p,q,\tilde{p},\tilde{q}} N^{-2s} \left( \sum_\mu \Vert a_\mu(t) \Vert^{q^\prime}_{L_t^{q^\prime}(\mathbb{R})} \right)^{1/q^\prime} \left( \sum_\nu \Vert b_\nu(t) \Vert^{\tilde{q}^\prime}_{L_t^{\tilde{q}^\prime}(\mathbb{R})} \right)^{1/\tilde{q}^\prime} \\
&=  C_{n,p,q,\tilde{p},\tilde{q}} N^{-2s} \left\Vert \left( \sum_\mu a_\mu(t)^{q^\prime} \right)^{1/q^{\prime}} \right\Vert_{L_t^{q^\prime}(\mathbb{R})} \left\Vert \left( \sum_\nu b_\nu(t)^{\tilde{q}^\prime} \right)^{1/\tilde{q}^{\prime}} \right\Vert_{L_t^{\tilde{q}^\prime}(\mathbb{R})}.
\end{split}
\end{equation*}
Finally, we use the embeddings $\ell^q \hookrightarrow \ell^p$ and $\ell^{\tilde{q}} \hookrightarrow \ell^{\tilde{p}}$, which gives
\begin{equation*}
\begin{split}
|B^N(F,G)| &\leq C_{n,p,q,\tilde{p},\tilde{q}} N^{-2s} \left\Vert \left( \sum_\mu a_\mu(t)^{p^\prime} \right)^{1/p^{\prime}} \right\Vert_{L_t^{q^\prime}(\mathbb{R})} \left\Vert \left( \sum_\nu b_\nu(t)^{\tilde{p}^\prime} \right)^{1/\tilde{p}^{\prime}} \right\Vert_{L_t^{\tilde{q}^\prime}(\mathbb{R})} \\
&= C_{n,p,q,\tilde{p},\tilde{q}} N^{-2s} \left\Vert \Vert F(t) \Vert_{Z_{p^\prime}} \right\Vert_{L_t^{q^\prime}} \left\Vert \Vert G(t) \Vert_{Z_{\tilde{p}^\prime}} \right\Vert_{L_t^{\tilde{q}^\prime}},
\end{split}
\end{equation*}
when in the last step we have used the properties of atomic decomposition.
\end{proof}
\section{Applications}
\label{section:applications}
Next, we give two instances of generalized Strichartz estimates: First, we see that requiring the wave-functions to be spherically symmetric yields generalized Strichartz estimates and further, we see that taking spherical averages yields generalized Strichartz estimates.
As future extended decay parameters we set
\begin{align*}
 \sigma^\prime(a,n) =\left\{\begin{array}{cl} n-1, &\; \mbox{if } a =1, \\
 \frac{2n-1}{2}, &\; \mbox{if } a > 1. \end{array} \right.
\end{align*}
Note that the estimates found after taking spherical averages imply the estimates found after requiring spherical symmetry for Schr\"odinger-like equations. But since we would like to stress the existence of a unified framework which is built up in Section \ref{section:preliminaries} allowing one to prove inhomogeneous estimates from homogeneous estimates we chose to present the results separately, also with a view towards possible generalizations with respect to the dispersion relation since the Strichartz estimates for spherically symmetric wave-functions are known for a much larger class than for the Schr\"odinger-like equations (cf. \cite{Cho2013}).
\subsection{Inhomogeneous estimates found after taking spherical symmetry}
\label{subsection:sphericalSymmetry}
In \cite{Cho2013} Cho and Lee showed that penalizing anisotropic propagation results in the following additional homogeneous Strichartz estimates for dispersion relations respecting spherical symmetry and their results imply the following corollary:
 \begin{theorem}[Special case of {\cite[Theorem~1.2.,~p.~997]{Cho2013}} \\ and  {\cite[Section~4.5.,~pp.~1017-1018]{Cho2013}}]
 \label{thm:radStrEst}
 $\,$ \\ Let $q,p \geq 2, \; a \geq 1$ and  $n \geq 2$ for $a > 1$ and $n \geq 3$ for $a =1$ and suppose that 
 \begin{equation*}
 \label{eq:conditionsOnQRadStrEst}
 \sigma(a,n) \left( \frac{1}{2} - \frac{1}{p} \right) < \frac{1}{q} < \sigma^\prime(a,n) \left( \frac{1}{2} - \frac{1}{p} \right). 
 \end{equation*}
Then we find the estimate 
\begin{equation*}
 \left\Vert P_N e^{it D^a} u_0 \right\Vert_{L_t^q L_x^p} \lesssim_{n,p,q} N^{-s} \Vert u_0 \Vert_{L^2(\mathbb{R}^n)} 
\end{equation*}
to hold for spherically symmetric $u_0$ where $s = - \frac{n}{2} + \frac{n}{p} + \frac{a}{q} $.
\end{theorem}
The special case $a=1$ was already covered in \cite{Sterbenz2005}, testing the estimates against Knapp-type examples one finds the range of integrability coefficients to be sharp up to endpoints.
We argue that Theorem \ref{thm:radStrEst} gives rise to generalized Strichartz estimates. The range spaces are $L^p$-spaces of spherically symmetric functions, which certainly have the compatibility property, also observe the identification:
\begin{equation*}
\mathcal{L}_r^p = L^p((0,\infty), r^{n-1} dr) \leftrightarrow \left\{ f \in L^p(\mathbb{R}^n) \, | \, f \mbox { spherically symmetric} \right\}
\end{equation*}
 Furthermore, the generalized Strichartz estimates admit the generalized dispersive estimate because the propagator respects spherical symmetry. Altogether, we find an instance of Theorem \ref{thm:globalInhomogeneousEstimates} to hold for spherically symmetric wave-functions with extended decay parameter defined above.\\
 We observe that the range of homogeneous estimates in Theorem \ref{thm:radStrEst} remains valid for general initial data if one considers norms of the initial data taking into account the regularity in the spherical coordinates. Performing an additional Littlewood-Paley decomposition in the spherical coordinates, one will be able to prove inhomogeneous estimates in the same range with derivative loss in the spherical coordinates.
 \subsection{Inhomogeneous estimates found after taking spherical averages}
 \label{subsection:sphericalAverages}
  For the wave equation and Schr\"odinger-like equations it is known, that the additional Strichartz estimates, which exist for spherically symmetric initial data, become also possible for general initial data when one requires a lower angular integrability of the corresponding free solutions.
The corresponding theorem on homogeneous estimates states as follows:
 \begin{theorem}[{\cite[Theorem~1.4.,~p.~4]{Jiang2012}}, {\cite[Theorem~1.1.,~p.~3]{Guo2014}}]
 \label{thm:sphAvgStrichartzEstimates}
 Let $a \geq 1, \; n \geq 3,\; q,p \geq 2$ and suppose that 
 \begin{equation*}
 \label{eq:conditionsOnQRadStrEstC}
 \sigma(a,n) \left( \frac{1}{2} - \frac{1}{p} \right) < \frac{1}{q} < \sigma^\prime(a,n) \left( \frac{1}{2} - \frac{1}{p} \right). 
 \end{equation*}
We find the estimate 
\begin{equation}
\label{eq:sphAvgStrEst}
 \left\Vert P_N e^{it D^a} u_0 \right\Vert_{L_t^q \mathcal{L}_r^p L_\omega^2} \lesssim_{n,p,q} N^{-s} \Vert u_0 \Vert_{L^2}
\end{equation} 
 to hold for any $N \in 2^{\mathbb{Z}}$ with $s = - \frac{n}{2} + \frac{n}{p} + \frac{a}{q} $.
 \end{theorem}
 We show that the above theorem is another instance of generalized Strichartz estimates with the range spaces $\mathcal{L}_r^p L_\omega^2$. The embedding $\mathcal{L}_r^p L_\omega^2 \hookrightarrow \mathcal{S}^\prime(\mathbb{R}^n)$ is clear from H\"older's inequality like the vector-valued $L^p$-structure.
 We still have to show compatibility with respect to frequency localization, which we do with the following lemma, extending Young's inequality. Let $\mu$ be the Haar measure on $SO(n)$, and denote $L^q_A = L^q(SO(n), \mu)$.\\
  The following identities are given in \cite[Lemma~3.1.,~Lemma~3.2.,~p.~255]{GuoLee2014}.\\
   First, we note that for any $p,q \in [1,\infty]$
  \begin{equation*}
  \label{eq:mixedNormIdentity}
  \Vert f \Vert_{\mathcal{L}_r^p L_\omega^q} \sim_{n,p,q} \Vert f \Vert_{L_x^p L_A^q},
  \end{equation*}
  which is straight-forward. This identification allows us to easily prove the following extension to Young's inequality in Euclidean space:
 \begin{lemma}[Young's inequality with mixed norms]
 \label{lem:YoungMixed}
$\,$ \\Suppose that $1 \leq p, q, p_1, p_2, q_1, q_2 \leq \infty, \; \frac{1}{q} = \frac{1}{q_1} + \frac{1}{q_2}, \; 1+ \frac{1}{p} = \frac{1}{p_1} + \frac{1}{p_2}.$ Then we find the following estimate to hold:
\begin{equation*}
\Vert f * g \Vert_{\mathcal{L}_r^p L_\omega^q} \leq C_{n,p,q} \Vert f \Vert_{\mathcal{L}_r^{p_1} L_\omega^{q_1}} \Vert g \Vert_{\mathcal{L}_r^{p_2} L_\omega^{q_2}}.
\end{equation*}
 \end{lemma}
We note that in the last step of the proof given in \cite{GuoLee2014} one could use weak Young's inequality. Since the Riesz-potential and the Bessel-potential are given by convolution with a spherically symmetric function, we conclude that Sobolev embedding remains valid in the $\mathcal{L}_r^p L_\omega^2$-spaces.\\
Specifically, we find from the above lemma that frequency localization yields a continuous operator in the $\mathcal{L}_r^p L_\omega^2$-spaces because frequency localization can be perceived as convolution with a spherically symmetric Schwartz function.\\ 
 The generalized dispersive estimate follows from two applications of H\"older's inequality and we conclude that spherically averaged estimates yield another instance of generalized Strichartz estimates and we find another instance of Theorem \ref{thm:globalInhomogeneousEstimates} to hold.\\
However, the application of H\"older's inequality to find the dispersive estimate in the $\mathcal{L}_r^p L_\omega^2$-spaces produces slack in the results. Alternatively, when one looks for the local estimates, one can directly interpolate with estimate \eqref{eq:genDispEst} and use H\"older's inequality in the spherical coordinates afterwards.\\
\subsection{Application to the fractional Schr\"odinger equation with potential}
\label{subsection:fractionalSEQApplication}
Finally, we give a more sophisticated application of the additional inhomogeneous estimates. Note the trivial application that the additional inhomogeneous estimates allow us to bind the weak solution to an inhomogeneous equation with zero-initial value in certain $L_t^q L_x^p$-norms, in which the weak solution to the homogeneous equation with non-vanishing initial value can't be bounded in general.
In \cite{Cho2016} had been considered the fractional Schr\"odinger equation with spherically symmetric initial data $u_0$ and potential $V$, where $1<a<2$:
 \begin{equation}
 \label{eq:fractionalEquationwithPotential}
 \left\{\begin{array}{cl}
 i \partial_t u(t,x) + D^a u(t,x) = V(t,x) u(t,x), \; (t,x) \in (\mathbb{R},\mathbb{R}^n), \\
 u(0,\cdot) = u_0  \end{array} \right.
 \end{equation}
 The main ingredient to the proof of well-posedness with initial data below $L^2$ are inhomogeneous estimates, which do not follow from the homogeneous estimates and the Christ-Kiselev lemma.\\
 We see how employing the additional inhomogeneous estimates allow us to drop the assumptions on spherical symmetry, but we have to require some angular regularity for the potential and slightly more Sobolev regularity for the potential and the initial data; we stay below $L^2$ though.
 We proceed by sketching the proof of \cite[Theorem~1.2.,~p.~1908]{Cho2016}. We shall see how additional inhomogeneous estimates make the proof possible. When we want to drop the assumptions on spherical symmetry, we make use of the estimates provided by Corollary \ref{cor:inhomogeneousEstimatesApplication}.
  \begin{proof}[Proof of Corollary \ref{cor:inhomogeneousEstimatesApplication}]
 We show that the conditions from Theorem \ref{thm:globalInhomogeneousEstimates} are fulfilled; first by checking the conditions on the diagonal $q=\tilde{q}$. If the inequalities hold on the diagonal strictly, the claim follows from continuous dependence.\\
 We have $$\mu(a) = \frac{(a/2) \left( \sigma_1/2 + \sigma_2/2 - n/2 \right)}{n(1-a/2) + \left( (a \sigma_1 -n)/(2q) + (a \sigma_2 -n)/(2q) \right)} \mbox{ and } q=\frac{2(n+a)}{n}. $$
  For $a=2$ this gives $\mu(a=2) = \frac{q}{2}$, and we find that the inequalities $\mu \leq p/2 $ and $\mu \leq \tilde{p}/2$ from Theorem \ref{thm:globalInhomogeneousEstimates} hold with equality, when all of the other inequalities are strict.\\
  For the derivatives of $\mu$ and $q$ we find
  $$ \mu^\prime(a=2) = \frac{q}{4} - \left( \frac{1}{2(\sigma_1 + \sigma_2-n)} \right) \left( 2+ \sigma_1 +  \sigma_2 - \frac{n}{2} q^2 \right)  \mbox{ and } 
  q^\prime(a=2)=\frac{2}{n}.$$
  We find that $\mu$ decreases much faster than $q$ as we lower $a$ starting from $a=2$ for $\sigma_1, \sigma_2 \downarrow n/2$, which yields the claim.
 \end{proof}
 Connecting these estimates which will be employed in the proof to certain regions in Figure \ref{fig:inhomogeneousEstimatesFigure} will clarify how additional inhomogeneous estimates establish well-posedness with negative Sobolev regularity.
 \begin{figure}[h]
	\begin{tikzpicture}[>=stealth,scale=0.55]
 	\node[shape=circle,inner sep=2pt, minimum size=2pt,label=right:$A$,draw]		(A) at (5,7) {};
 	\node[shape=circle,inner sep=2pt, minimum size=2pt,label=left:$B$,draw]		(B) at (4.5,4.5) {};
 	\node[shape=circle,inner sep=2pt, minimum size=2pt,label=right:$C$,draw]		(C) at (6,6) {};
 	\node[shape=circle,inner sep=2pt, minimum size=2pt,label=below:$D$,draw]		(D) at (7,0) {};
 	\draw[->, very thick]	(0,0) -- (D) -- (8,0) node[anchor=north] {$\frac{1}{p} \, \left( \frac{1}{\tilde{p}} \right)$} ;
 	\draw[->, very thick]	(0,0) node[anchor=east] {$O$} -- (0,8) node[anchor=east] {$\frac{1}{q} \, \left( \frac{1}{\tilde{q}} \right)$};
 	\draw[dashed]	(0,0) -- (B) -- (C);
 	\draw (D) -- (A);
	\draw (A) -- (B) -- (D) -- (C) -- (A);
 	\draw[dashed] (A) -- (0,7) node[anchor=east] {$\frac{1}{2}$};
	\end{tikzpicture}
		\caption{ We give a pictorial representation similar to \cite[Figure~1,~p.~1907]{Cho2016}. 
In the setting of \cite[Corollary~1.,~p.~1907]{Cho2016} we find $A=(\frac{n-a}{2n},\frac{1}{2}), \, B= (\frac{n}{n+a}-\frac{n}{2(n+1)},\frac{n}{n+a}-\frac{n}{2(n+1)}), \, C=(\frac{n}{2(n+1)},\frac{n}{2(n+1)}), \, D=(\frac{1}{2},0)$, where the open line $\overline{BC}$ corresponds to the range from \cite[Theorem~1.2.,~p.~1908]{Cho2016} and the closed line $\overline{AD}$ corresponds to estimates found from factorization and application of the Christ-Kiselev lemma.}
	\label{fig:inhomogeneousEstimatesFigure}
\end{figure}
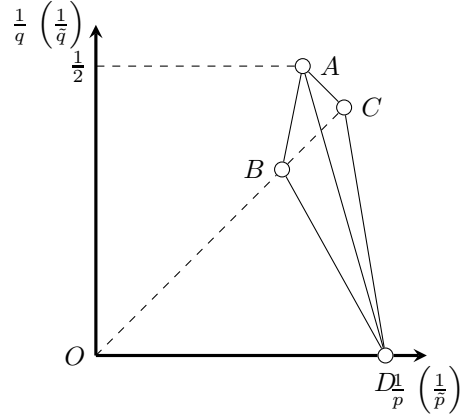 
  \begin{proof}[Proof of {\cite[Theorem~1.2.,~p.~1908]{Cho2016}} and dropping spherical symmetry]$\,$ \\
 We will see that the solution mapping
\begin{equation*}
 \Phi_{u_0}(u) = e^{itD^a} u_0 -i \int_0^t e^{i(t-s)D^a} (Vu)(s) ds
 \end{equation*}
 is a contraction mapping if we choose an adequate resolution space. The proof from \cite{Cho2016} can be divided up into the following two steps:
 \begin{enumerate}
 \item[(i)] Establishing local well-posedness in some $L_t^q([0,\tau],L_x^p)$-space for a small $\tau >0$,
 \item[(ii)] Iterating the process a finite number of times after showing that the Sobolev regularity is conserved. Since we show that we control the $L_t^\infty \dot{H}^\gamma_x$-norm continuity follows from the usual approximation argument.
 \end{enumerate}
 When we want to drop the assumptions on spherical symmetry, we work in a space given by the norm
 \begin{equation*}
\Vert F \Vert_{Y} = \sum_{N} \Vert P_N F \Vert_{L_t^{q} \mathcal{L}_r^{p} L_\omega^2},
\end{equation*}
when the solution mapping $\Phi_{u_0}$ remains unchanged of course. In this case working with a $1$-norm has practical benefits. Note that we have from the triangle inequality
$ \Vert u \Vert_{L_t^{q} \mathcal{L}_r^{p} L_\omega^2} \leq \Vert u \Vert_Y$. In the second step we iterate with respect to $L_t^\infty \dot{B}^\gamma_{2,1}$.\\
For the homogeneous part of the solution mapping we find from the homogeneous Strichartz estimates
\begin{equation*}
\Vert e^{itD^a} u_0 \Vert_{Y} \lesssim_{n,p,q} \Vert u_0 \Vert_{\dot{B}^\gamma_{2,1}}
\end{equation*}
and for the inhomogeneous part we find for a frequency localized component by virtue of the estimates found in Theorems \ref{thm:globalInhomogeneousEstimates} under the associated assumptions
\begin{equation*}
\Vert P_N \int_{0}^t e^{i(t-s)D^a} (Vu)(s) ds \Vert_{L_t^{q} \mathcal{L}_r^{p} L_\omega^2} \lesssim_{n,p,q,\tilde{p},\tilde{q}} \Vert \tilde{P}_N (Vu) \Vert_{L_t^{\tilde{q}^\prime} \mathcal{L}_r^{\tilde{p}^\prime} L_\omega^2}.
\end{equation*}
Ad (i): We find
\begin{equation*}
\Vert \Phi_{u_0}(u) \Vert_{Y} \lesssim_{n,p,q,\tilde{p},\tilde{q}} \Vert u_0 \Vert_{\dot{B}^\gamma_{2,1}} + \sum_N \Vert \tilde{P}_N (Vu) \Vert_{L_t^{\tilde{q}^\prime}([0,\tau], \mathcal{L}_r^{\tilde{p}^\prime} L_\omega^2)}
\end{equation*}
where we make use of the homogeneous estimates, which demands
\begin{equation}
\label{eq:linePQGamma}
\frac{a}{q} + \frac{n}{p} = \frac{n}{2} - \gamma 
\end{equation} 
to bind the homogeneous part; the admissibility follows from requirements on $a$ and $\gamma$. Further, $(q,p,\tilde{q},\tilde{p})$ must be in the range of the additional inhomogeneous estimates, which gives
\begin{equation}
\label{eq:linetildePtildeQGamma}
\frac{a}{\tilde{q}} + \frac{n}{\tilde{p}} = \frac{n}{2} + \gamma
\end{equation}
from plugging in the scaling condition. The requirement that \eqref{eq:linePQGamma} has non-empty intersection with the triangle $\triangle(ADC)$, with the lines $\overline{AC}$ and $\overline{CD}$ excluded, leads to conditions on $\gamma$ and $q$. In the case of spherical symmetry, these are described by \cite[Eq.~(20),~(22),~p.~1909]{Cho2016}.\\
The requirement that \eqref{eq:linetildePtildeQGamma} has non-empty intersection with the triangle $\triangle(ABC)$, with the lines $\overline{AB}$ and $\overline{BC}$ excluded, leads to an additional condition on $\frac{1}{\tilde{q}}$ (cf. \cite[Eq.~(23),~p.~1909]{Cho2016}). In the following we adapt the notation from \cite{Cho2016}.\\
We decompose $P_N(V u) = P_N((P_{<N/8} V) u) + P_N((P_{\geq N/8} V) u)$ and for the first term we note that we can freely replace $u$ with $P_{N/8<\cdot<8N} u$ due to impossible frequency interactions and we find by making use of H\"older's inequality and Sobolev embedding on the sphere
\begin{equation*}
\begin{split}
\Vert P_N ((P_{<N/8} V)u) \Vert_{L_t^{\tilde{q}^\prime} \mathcal{L}_r^{\tilde{p}^\prime} L_\omega^2} &\leq C_{n} \Vert (P_{<N/8} V) (P_{N/8<\cdot<8N} u) \Vert_{L_t^{\tilde{q}^\prime} \mathcal{L}_r^{\tilde{p}^\prime} L_\omega^2} \\
&\leq C_{n} \Vert P_{<N/8} V \Vert_{L_t^r \mathcal{L}_r^w L_\omega^{\infty}} \Vert P_{N/8<\cdot<8N} u \Vert_{L_t^{q} \mathcal{L}_r^{p} L_\omega^2} \\
&\leq C_{n,p,w,\alpha} \Vert \Lambda_\omega^{\alpha} V \Vert_{L_t^r L_x^w} \sum_{M \sim N} \Vert P_M u \Vert_{L_t^{q} \mathcal{L}_r^{p} L_\omega^2},
\end{split}
\end{equation*}
whenever $\alpha > \frac{n-1}{w}$.\\
For the second term $(P_{\geq N/8} V) u$ we distinguish between three frequency regions:\\
In case of $N \ll 1$ we make use of Bernstein's inequality, which we state for convenience for spaces with mixed norms:
$$ \Vert P_N g \Vert_{\mathcal{L}_r^q L_\omega^2} \leq C_{n,p,q} N^{\frac{n}{p}-\frac{n}{q}} \Vert P_N g \Vert_{\mathcal{L}_r^p L_\omega^2}, $$
whenever $1 \leq p \leq q \leq \infty$. For the proof one can follow along the lines of the proof of the common variant, but making use of Lemma \ref{lem:YoungMixed} instead of usual Young's inequality. 
This gives for some small $\varepsilon^\prime >0$
\begin{equation*}
\begin{split}
\Vert P_N ((P_{\geq N/8} V) u) \Vert_{L_t^{\tilde{q}^\prime} \mathcal{L}_r^{\tilde{p}^\prime} L_\omega^2} &\leq C_{n,p,w,\varepsilon^\prime} N^\varepsilon \Vert (P_{\geq N/8} V) u \Vert_{L_t^{\tilde{q}^\prime} \mathcal{L}_r^{\tilde{p}^\prime - \varepsilon^\prime} L_\omega^2} \\
&\leq C_{n,p,w,\varepsilon^\prime} N^\varepsilon \Vert P_{ \geq N/8} V \Vert_{L_t^r \mathcal{L}_r^{w-\varepsilon^\prime} L_\omega^\infty} \Vert u \Vert_{L_t^{q} \mathcal{L}_r^{p} L_\omega^2} \\
 &\leq C_{n,p,w,\varepsilon^\prime,\alpha} N^\varepsilon \Vert \Lambda_\omega^\alpha V \Vert_{L_t^r L_x^{w-\varepsilon^\prime}} \Vert u \Vert_Y.
\end{split}
\end{equation*}
For $N \sim 1$ we make use of the crude estimate, which follows from taking out the operator norms of the frequency projectors:
\begin{equation*}
\begin{split}
\Vert P_N ((P_{\geq N/8} V) u) \Vert_{L_t^{\tilde{q}^\prime} \mathcal{L}_r^{\tilde{p}^\prime} L_\omega^2} 
&\leq C_{n,p,w,\alpha} \Vert V \Vert_{L_t^r \mathcal{L}_r^{w} L_\omega^\infty} \Vert u \Vert_{L_t^{q} \mathcal{L}_r^{p} L_\omega^2} \\
&\leq C_{n,p,w,\alpha} \Vert \Lambda_\omega^\alpha V \Vert_{L_t^r L_x^w} \Vert u \Vert_Y,
\end{split}
\end{equation*}
which is still acceptable because we only have to sum finitely many of these pieces.\\
For $N \gg 1$ we make use of the following Bernstein inequality:
$$ \Vert P_{\geq N} V \Vert_{\mathcal{L}_r^p L_\omega^2} \lesssim_{n,p,s} N^{-s} \Vert P_{\geq N} D^s V \Vert_{\mathcal{L}_r^p L_\omega^2}, $$
which holds, whenever $s \geq 0$ and $1 \leq p \leq \infty$. This gives
\begin{equation*}
\begin{split}
\Vert P_N ((P_{\geq N/8} V) u) \Vert_{L_t^{\tilde{q}^\prime} \mathcal{L}_r^{\tilde{p}^\prime} L_\omega^2} 
&\leq C_{n} \Vert P_{\geq N/8} V \Vert_{L_t^r \mathcal{L}_r^{w} L_\omega^\infty} \Vert u \Vert_{L_t^{q} \mathcal{L}_r^{p} L_\omega^2} \\
&\leq C_{n,p,w,\varepsilon,\alpha} N^{-\varepsilon} \Vert \Lambda_\omega^\alpha V \Vert_{L_t^r W_x^{\varepsilon,w}} \Vert u \Vert_{Y}.
\end{split}
\end{equation*}
We find the solution mapping to be contracting if $\Lambda_\omega^{\alpha} V \in L_t^r L_x^{w-\varepsilon} \cap L_t^r W_x^{\varepsilon,w}$ with $\varepsilon >0, \; \alpha > \frac{n-1}{w-\varepsilon}$.\\
Ad (ii): For the homogeneous part we find again from the energy estimate
\begin{equation*}
\Vert e^{itD^a} u_0 \Vert_{L_t^\infty \dot{B}^\gamma_{2,1}} = \left\Vert \sum_N N^\gamma \Vert P_N e^{it D^a} u_0 \Vert_{L_x^2} \right\Vert_{L_t^\infty}
\leq \sum_N N^\gamma \Vert P_N u_0 \Vert_{L_x^2} = \Vert u_0 \Vert_{\dot{B}^\gamma_{2,1}}.
\end{equation*}
For the inhomogeneous part we can also follow the strategy from \cite{Cho2016} using the same notation:
\begin{equation*}
\begin{split}
N^\gamma \Vert P_N \int_0^t e^{i(t-s)D^a} (Vu)(s) ds \Vert_{L_t^\infty L_x^2}
&\leq C_{n,\tilde{u},\tilde{v}} N^\gamma \Vert P_N (Vu) \Vert_{L_t^{\tilde{u}^\prime} \mathcal{L}_r^{\tilde{v}^\prime} L_\omega^2} \\
&\leq C_{n,\tilde{u},\tilde{v}} \Vert P_N (Vu) \Vert_{L_t^{\tilde{u}^\prime} \mathcal{L}_r^{b^\prime} L_\omega^2}.
\end{split}
\end{equation*}
Again we decompose $P_N(Vu) = P_N ((P_{<N/8} V) (P_{N/8<\cdot<8N} u)) + P_N ((P_{\geq N/8} V) u)$ and for the first term we find:
\begin{equation*}
\begin{split}
&\Vert P_N ((P_{<N/8} V) (P_{N/8<\cdot<8N} u)) \Vert_{L_t^{\tilde{u}^\prime} \mathcal{L}_r^{b^\prime} L_\omega^\infty} \\
&\leq C_{n,p,w,\alpha} \Vert \Lambda_\omega^\alpha V \Vert_{L_t^r L_x^w} \sum_{M \sim N} \Vert P_M u \Vert_{L_t^q \mathcal{L}_r^p L_\omega^2}, 
\end{split}
\end{equation*} 
where the second factor is controlled by $\Vert u \Vert_Y$ after summing over $N$.\\
The second term will be treated like in the first part of the proof:
For $N \ll 1$ we can employ a Bernstein inequality and find by the same means of the first part
$$ \Vert P_N((P_{\geq N/8} V)u) \Vert_{L_t^{\tilde{u}^\prime} \mathcal{L}_r^{b^\prime} L_\omega^2} \leq C_{n,p,w,\varepsilon^\prime,\alpha} N^\varepsilon \Vert \Lambda_\omega^{\alpha} V \Vert_{L_t^r L_x^{w-\varepsilon^\prime}} \Vert u \Vert_Y. $$
For $N \sim 1$ we make use of the rough estimate from the first part to find
$$\Vert P_N((P_{\geq N/8} V)u) \Vert_{L_t^{\tilde{u}^\prime} \mathcal{L}_r^{b^\prime} L_\omega^2} \leq C_{n,p,w,\alpha} \Vert \Lambda_\omega^\alpha V \Vert_{L_t^r L_x^w} \Vert u \Vert_Y $$
and for $N \gg 1$ we find 
$$\Vert P_N((P_{\geq N/8} V)u) \Vert_{L_t^{\tilde{u}^\prime} \mathcal{L}_r^{b^\prime} L_\omega^2} \leq
C_{n,p,w,\varepsilon,\alpha} N^{-\varepsilon} \Vert \Lambda_\omega^\alpha V \Vert_{L_t^r W_x^{\varepsilon,w}} \Vert u \Vert_Y,$$
which means that we need no additional requirements on $V$ to perform the iteration.
\end{proof}
\section*{Acknowledgements}
This article forms part of the author's master's thesis. The author would like to thank his thesis supervisor Sebastian Herr for numerous helpful comments related to this work. Financial support by the German Science Foundation DFG (IRTG 2235) is gratefully acknowledged.

\end{document}